\documentclass[12pt,a4paper,reqno]{amsart}
\usepackage{amsmath,amssymb,graphics,epsfig,color,enumerate,mathrsfs, url}
\usepackage{mathtools}
\mathtoolsset{showonlyrefs,showmanualtags}
 \usepackage{tikz}

\usepackage[T1]{fontenc}
\usepackage[utf8]{inputenc}

\usepackage{dsfont}  \usepackage{verbatim}\textwidth= 14. cm
\definecolor{refkey}{gray}{.75}
\definecolor{labelkey}{gray}{.5}

\newtheorem{Theorem}{Theorem}[section]

\newtheorem{Lemma}[Theorem]{Lemma}

\newtheorem{Corollary}[Theorem]{Corollary}
\newtheorem{Remark}[Theorem]{Remark}

\newtheorem{Claim}[Theorem]{Claim}
\newtheorem{Definition}[Theorem]{Definition}

 \definecolor{darkgreen}{rgb}{0,0.6,0}

\definecolor{light}{gray}{0.9}

\newcommand{\verde}{\textcolor{black}}  
\newcommand{\rosso}{\textcolor{black}}
\newcommand{\rrr}{\textcolor{black}}

\newcommand{\rot}{\textcolor{black}} 
\newcommand{\grun}{\textcolor{black}} 
\newcommand{\ciak}{\textcolor{black}} 

\newcommand{\cA}{\ensuremath{\mathcal A}}
\newcommand{\cB}{\ensuremath{\mathcal B}}
\newcommand{\cC}{\ensuremath{\mathcal C}}

\newcommand{\cF}{\ensuremath{\mathcal F}}
\newcommand{\cG}{\ensuremath{\mathcal G}}

\newcommand{\cI}{\ensuremath{\mathcal I}}

\newcommand{\cL}{\ensuremath{\mathcal L}}
\newcommand{\cM}{\ensuremath{\mathcal M}}

\newcommand{\cP}{\ensuremath{\mathbb P}}  

\newcommand{\cV}{\ensuremath{\mathcal V}}

\newcommand{\bbE}{{\ensuremath{\mathbb E}} }

\newcommand{\bbG}{{\ensuremath{\mathbb G}} }

\newcommand{\bbI}{{\ensuremath{\mathbb I}} }

\newcommand{\bbL}{{\ensuremath{\mathbb L}} }

\newcommand{\bbN}{{\ensuremath{\mathbb N}} }

\newcommand{\bbP}{{\ensuremath{\mathbb P}} }
\newcommand{\bbQ}{{\ensuremath{\mathbb Q}} }
\newcommand{\bbR}{{\ensuremath{\mathbb R}} }

\newcommand{\bbZ}{{\ensuremath{\mathbb Z}} }

\let\a=\alpha \let\b=\beta   \let\d=\delta  \let\e=\varepsilon
 \let\g=\gamma     \let\k=\kappa  \let\l=\lambda
      \let\o=\omega      
  \let\s=\sigma \let\t=\tau   
  
\let\D=\Delta   \let\G=\Gamma   
\let\O=\Omega      
\newcommand{\da}{\downarrow}

\newcommand{\toup}{\rightharpoonup}

\newcommand{\be}{\begin{equation}}
\newcommand{\en}{\end{equation}}
\newcommand{\bem}{\begin{multline}}
\newcommand{\enm}{\end{multline}}
\newcommand{\bes}{\begin{equation*}}
\newcommand{\ens}{\end{equation*}}

\author[A.~Faggionato]{Alessandra Faggionato}
\address{Alessandra Faggionato. Department of Mathematics, University La Sapienza, 
  P.le Aldo Moro 2, 00185 Rome, Italy}
\email{faggiona@mat.uniroma1.it}

\newcommand{\ra}{\rangle}
\newcommand{\la}{\langle}

\title[An  ergodic theorem with weights and applications]{\rot{An  ergodic theorem with weights and  applications to random measures, homogenization and  hydrodynamics}}

\begin{document}

\begin{abstract}
We prove a multidimensional  ergodic theorem with weighted averages for the action of the group $\mathbb{Z}^d$ on a probability space. At level $n$ weights are of the form $n^{-d} \psi(j/n)$, $ j\in \mathbb{Z}^d$,  for real  functions $\psi$  decaying suitably fast.  We discuss applications to random measures and to quenched stochastic homogenization of  random walks  on simple point processes with  long-range random jump rates, allowing to remove the technical Assumption (A9) from  \cite[Theorem~4.4]{Fhom1}. This last result concerns also  some semigroup and resolvent convergence particularly relevant for the derivation of the quenched   hydrodynamic limit of interacting particle systems   via homogenization and duality. As a consequence we show that  also the quenched hydrodynamic limit of the symmetric simple exclusion process on point processes stated in   \cite[Theorem~4.1]{F_SEP} remains valid when removing  the above mentioned Assumption (A9).


\end{abstract}

\maketitle

\hspace{8cm}
In memory of Francis Comets.

\bigskip
\emph{
  I desire to thank Francis Comets for his kindness and cheerfulness towards his colleagues and also young people taking
  their first steps in research (this was my case when we met the first time). I  benefited from  both his scientific and human qualities. Thanks!
}

\section{Introduction}
The multidimensional ergodic theorem with $\bbZ^d$--action and arithmetic averages states that, given  $d$ commuting measure-preserving and bijective  maps  $T_1,T_2,\dots, T_d$    on a probability space $(\O, \cF,\bbP)$ and setting  $T^j:= T_1^{j_1} \circ T_2^{j_2} \circ \cdots \circ T_d^{j_d}$ for $j=(j_1,j_2,\dots, j_d) \in \bbZ^d$, the  sequence 
\be\label{onde}
\frac{1}{|I_n|} S_{I_n} f := \frac{1}{|I_n|}  \sum_{j\in I_n} f\circ T^j
\en
converges as $n\to+\infty$  a.s. and in $L^p$ to $\bbE[f|\cI]$ for any $f\in L^p:=L^p(\O,\cF,\bbP)$, $\rot{p\in[1,+\infty)}$. Above $\cI$ is the $\s$--algebra of invariant measurable sets and $I_1\subset I_2 \subset I_3\subset \cdots $ is \rrr{any} increasing sequence  of boxes  with union all $\bbZ^d$. The above result is due to Tempelman \cite{T1972}  (see e.g. \cite[Theorem.~2.8, Chapter~6]{K} and \cite[Theorem~2.6]{S}). The multidimensional ergodic theorem can be derived from the  maximal inequality (\rot{also called} dominated ergodic theorem), which reads  $\bbP( |I_n| ^{-1} S_{I_n} f\geq \a) \leq C \|f\|_1 /\a$ for any non-negative function $f\in  L^1$ \rrr{and $\a>0$}.
The above ergodic  theorem has been further generalized by considering more general bounded  sets $I_n$,  by replacing arithmetic averages by weighted averages \rot{(see below)} and also by considering other groups including of course $\bbR^d$ (see \cite{K,S,T1972,Tbook} and the Introduction of \cite{TS}).

We are interested here to a generalization of the above multidimensional ergodic theorem for more general averages, where the arithmetic  average $|I_n|^{-1} S_{I_n} f $ is replaced by   $\sum_{j\in \bbZ^d} c_{n,j} f\circ T^j$ for \rot{real weights $c_{n,j}$ non necessarily with bounded support in $j$}. We  use here  the term average in a more relaxed way, not imposing that $c_n:=\sum_{j\in \bbZ^d} c_{n,j}$ equals $1$ but just that $c_n$ converges as $n\to +\infty$ to some  \rot{finite constant}. For the applications motivating our search of such a generalization (see below),
the weights  $c_{n,j}$ have not  bounded support, but decay fast enough  as $|j|\to +\infty$.
We point out that the main technical \ciak{difficulty} comes from the the unbounded support, the fact that we deal with real weights does not make any effective difference.

\rot{Some ergodic results with 
 weights also with unbounded support  are discussed for $d=1$} in  \cite[Chapter~8]{K} (and references therein).   \rot{The ergodic theorems  in \cite[Chapter~8.1]{K} refer to $L^p$ convergence. These results are based on manipulations of sequences and e.g.  on spectral theory when working with the Hilbert space $L^2$ as   in \cite{HP}.  The a.s. convergence has been derived for suitable choices of weights in    \cite[Chapter~8.2]{K}, by using also Abel's transformation for series.}
  For weights with bounded support \rot{and $d$ generic} we mention the results in \cite[Section~7]{T1972} and \cite{Tbook}[Section~4, Chapter~6]. \rot{Finally, an ergodic theorem with  weights with unbounded support in the multidimensional case is provided by  \cite[Proposition~5.3]{TS}. This result implies that}  for a large class of non-negative measurable  functions  $\psi$ on $\bbR^d$  not necessarily with compact support (e.g. for $\psi$ measurable and  such that $0\leq \psi(x) \leq C (1+|x|)^{-\b}$ with $\b>d$) given $p\in (1,+\infty)$ and $f\in L^p$
 the integral $n^{-d}\int_{\bbR^d}  \psi(x/n) f\circ  T^x dx$ converges a.s. to  $\rrr{c(\psi)} \bbE[f|\cI]$, where $(T^x)_{x\in \bbR^d}$ is an action of the group $\bbR^d$ on the probability space,  $\cI$ is the $\s$--algebra of invariant sets and $c(\psi):=\int_{\bbR^d} \psi(x) dx$.

  \rot{We stress that  \cite[Proposition~5.3]{TS}   does not cover the case $p=1$ and this is a limitation in the applications motivating our investigation}.
 In Theorem \ref{erg_thm} below  we present a multidimensional ergodic theorem \rot{covering also the case $p=1$ and} implying the following. Given the action $(T^j)_{j\in \bbZ^d}$  of $\bbZ^d$ on the probability space  $(\O,\cF,\bbP)$ and given a  map  $\psi:\bbR^d\to\bbR$
with  $|\psi(x) |\leq C (1+|x|)^{-\b}$ and $\b>2d+2$ \rrr{(and satisfying some minor conditions)},    the weighted average 
$n^{-d}\sum_{j\in \bbZ^d} \psi(j/n) f\circ T^j$ converges a.s. and in $L^p$  as $n\to \infty$  to $\rrr{c(\psi)}\bbE[f|\cI]$ for any $f\in L^p$ with  $p\in [1,+\infty)$.
Although the critical exponent $2d+2$ could not be optimal, this result is enough for our applications  and covers the case $p=1$. \rot{Our proof is different from the one of  \cite[Proposition~5.3]{TS}. The derivation of  Theorem \ref{erg_thm}   relies on a maximal inequality (cf.~Theorem \ref{teo_max_in}) with its own interest. The proof of the maximal inequality is based on a suitable  covering procedure extending the one (similar to Vitali's covering lemma) used in the derivation of Tempelman's multidimensional ergodic theorem  (cf.~\cite[Theorem.~2.8, Chapter~6]{K}).    }

Our applications concern random measures, stochastic homogenization of random walks on simple point \rot{processes}  and \rrr{hydrodynamic} limits of interacting particle systems on simple point processes. 
Let us consider the group $\bbG=\bbZ^d$ or $\bbG=\bbR^d$ acting on $\bbR^d$ by Euclidean translations and on the probability space $(\O,\cF,\bbP)$. We assume that $\bbP$ is stationary and ergodic w.r.t. the $\bbG$-action.
Let $\mu_\o$ be a random locally finite  measure on $\bbR^d$ for which a natural covariant relation is satisfied under the  two above actions  (see Section \ref{appl_random_meas}). By calling $\mu^\e_\o$ the rescaled measure $\mu^\e_\o(A):=\e^d \mu_\o(\e^{-1} A)$,  we show that
\[\lim_{\e\da 0} \int _{\bbR^d} \psi (x )\mu^\e_\o(dx) = m \int_{\bbR^d}  \psi(x) dx\qquad \bbP\text{--a.s.}
\]
where $m$ is the intensity of the measure, assumed to be finite, in the following cases \rrr{where $\psi\in C(\bbR^d)$}: (i)   $|\psi(x) |\leq C (1+|x|)^{-\b}$ with $\b>2d+2$, (ii)  $|\psi(x) |\leq C (1+|x|)^{-\b}$ with $\b>d$ if in addition   for some $\a>1$  \rrr{ $\bbE[\mu_\o(A)^\a]<+\infty$   for any bounded Borel set $A$}. The above result has been derived using
a tail control related to Theorem \ref{erg_thm} for case (i)  and \cite[Proposition~5.3]{TS} for case (ii). Note that a priori the measure $\mu_\o$ is not uniformly bounded on  balls of fixed radius and density fluctuations can be present with balls with arbitrarly large mass. Hence the above result provides a control at infinity of these   fluctuations.
In Section \ref{appl_random_meas} we present also further progresses on random measures  (see Theorems \ref{teo_tartina} and \ref{prop_pizza}, Lemmas \ref{cinguetto} and \ref{cinguetto_bis} and Corollary \ref{lollo_bronzo}). For other ergodic results concerning random measures we mention \cite{DV,NZ1,NZ2} and references therein.

And finally we arrive at our starting motivation. In \cite{Fhom1} we have derived quenched stochastic homogenization results for random walks with long-range random jump rates on simple point processes  on $\bbR^d$ (assuming a stationary and ergodic action of the group $\bbG$). In \cite{F_SEP} we have derived the quenched hydrodynamic limit in path space for random walks as above but  interacting via site exclusion  when the rates are symmetric  (the so called symmetric simple exclusion process). Both \cite{F_SEP} and 
\cite{Fhom1}   aim to  universal results applicable to a large class of models.
The homogenization in \cite{Fhom1}  concerns also  the   convergence of 
the  $L^2$-Markov semigroup and resolvent of the random walk towards the corresponding objects of the Brownian motion with diffusion matrix given by \rot{twice} the effective homogenized matrix. A suitable form of convergence \rot{(cf. \eqref{marvel1},...,\eqref{ondinoB} below)}, also crucial to derive the above mentioned  hydrodynamic limit, is derived in \cite{Fhom1} under an additional assumption called (A9) in \cite{Fhom1} \rot{(and recalled in Section \ref{sec_rw})}, which allows to control at infinity  regions where the simple point process has many points.  \rot{Roughly, Assumption (A9) requires that  the number of points in unit boxes is uniformly bounded or satisfies  a suitable covariance decay}. As a consequence the same assumption appears in \cite{F_SEP} \rot{which relies on \cite{Fhom1}}.  

Starting from our results for random measures  we show that this assumption \rot{(A9)} is not necessary anymore, and  the control at infinity is assured by ergodicity  \rot{itself (which was already between the basic assumptions in \cite{F_SEP} and \cite{Fhom1}). This was not possible at the time of  \cite{F_SEP} and \cite{Fhom1} exactly because a result like our Theorem \ref{erg_thm} was missing}.
For more details we refer to Section \ref{sec_rw} and in particular to Theorem \ref{teo2} and Corollary \ref{scherzetto}. As a consequence, both in \cite{F_SEP} and \cite{Fhom1} Assumption (A9) can now be removed.

\medskip

{\bf Outline of the paper}: In Section \ref{sec_erg_thm} we present our multidimensional ergodic theorem with weighted averages (Theorem \ref{erg_thm})  and the associated maximal inequality (Theorem \ref{teo_max_in}). In Section \ref{appl_random_meas} we discuss some applications to random measures (Theorems \ref{teo_tartina} and \ref{prop_pizza}, Lemmas \ref{cinguetto} and \ref{cinguetto_bis} and Corollary \ref{lollo_bronzo}). In Section \ref{sec_rw} we discuss applications to \rot{stochastic homogenization of} random walks with  \ciak{random jump rates} on simple point processes   (Theorem~\ref{teo2}) and to the hydrodynamic limit of the symmetric simple exclusion process on simple point processes with random jump rates (Corollary~\ref{scherzetto}).  The remaining sections \rot{and the appendix} are devoted to  proofs 
\ciak{(for Theorem~\ref{teo_max_in} see Section~\ref{proof_max_in}, for  Theorem~\ref{erg_thm} see Section~\ref{proof_erg_thm}, for Theorem~\ref{teo_tartina} and Lemma~\ref{cinguetto} see Section~\ref{settimino}, for  Theorem~\ref{prop_pizza} and Lemma~\ref{cinguetto_bis} see Section~\ref{WTS}, for Theorem~\ref{teo2} see Section~\ref{squillano})}.
 \rot{The proofs for Section  \ref{appl_random_meas} (Section \ref{sec_rw}) rely on the results of Section \ref{sec_erg_thm} (Section  \ref{appl_random_meas}, respectively), but the proofs for each section   can be read independently. }

\section{An ergodic theorem with weighted averages}\label{sec_erg_thm}
We fix some basic notation.
We set $\bbR_+:=[0,+\infty)$ and  $\bbN_+:=\{1,2,3,\dots \}$.  We denote by $e_1,e_2,\dots, e_d$ the canonical basis of $\bbZ^d$.
We fix $\k\in [1,+\infty]$ and, given $x\in \bbR^d$,  we denote by $|x|$ the $\ell^\k$-norm of $x$ (in particular, $|x|$ is the \ciak{Euclidean} norm of $x$ when $\k=2$).

 Let  $T_1,T_2,\dots, T_d$  be $d$ commuting measure-preserving and bijective  maps on a probability space $(\O, \cF,\bbP)$.   We call  $\cI\subset \cF$ the $\s$--subalgebra given by the invariant sets, i.e.~$\cI:=\{ A\in \cF\,:\, T_k ^{-1} A =A \text{ for all } 1\leq k\leq d\}$. Moreover we set  $T^j:= T_1^{j_1} \circ T_2^{j_2} \circ \cdots \circ T_d^{j_d}$ for $j=(j_1,j_2,\dots, j_d) \in \bbZ^d$. In what follows we write $L^p$ for $L^p(\O, \cF,\bbP)$ \rosso{and we denote by $\bbE[\cdot]$ the expectation w.r.t.~$\bbP$.}


\begin{Definition}\label{goodness}
A function $\vartheta:\bbR_+\to \bbR_+$ is called $d$-\emph{good} if it is non-increasing and 
 $\sum_{m=0}^\infty m^{2d}\vartheta(m) \rho(m)^{-1}<+\infty$ for some positive summable function $\rho:\bbN\to (0,+\infty)$.
\end{Definition}
Trivially, given $c>0$ and $\b>2d+2$,  the function   $\vartheta(r) := c\,(1+ r)^{-\b} $  on $\bbR_+$ is $d$-good (take $\rho(m):=(1+m)^{-1-\d}$ with $\d>0$ small). \rot{The reader, interested just in the applications presented in the next sections, can neglect the concept of $d$--good function and simply take $\vartheta(r) := c\,(1+ r)^{-\b} $  with $\b>2d+2$ in Theorems \ref{teo_max_in} and \ref{erg_thm} below.}

\smallskip
\rot{Recall that  $|x|$ denotes the   $\ell^\k$-norm of $x$, where  $\k\in [1,+\infty]$.} Our first result is the following maximal inequality,   (see Section \ref{proof_max_in} for the proof):

\begin{Theorem}[Maximal Inequality] \label{teo_max_in}
For any function $\psi: \bbR^d\to \bbR_+$ such that $\psi(x) \leq \vartheta (|x|)$ for some $d$-good  function $\vartheta:\bbR_+\to \bbR_+$, 
 for any non-negative $f \in L^1 $ and for any $\a>0$, it holds 
\be\label{max_ineq}\bbP\Big( \sup _{n\geq 1} \frac{1}{n^d} \sum_{j\in \bbZ^d} \psi (j/n)f( T^j \o)>\a\Big)  \leq  \frac{C \|f\|_1}{\a}\,,
\en
where   $C=C(d,\vartheta,\rho,\k)$ is a suitable positive constant \verde{and $\rho$ is as in Definition \ref{goodness}}.
\end{Theorem}

The above maximal inequality is \rot{a form of dominated ergodic theorem \cite{K} and} is 
 the main tool to derive the following result (see Section \ref{proof_erg_thm} for the proof):
\begin{Theorem}[Ergodic theorem for weighted averages] \label{erg_thm}
   Fix a  function   $\psi: \bbR^d\to \bbR$ such that
\begin{itemize}
\item[(i)] $|\psi(x)| \leq \vartheta (|x|)$ for some $d$-good  function $\vartheta$;
\item[(ii)] $\lim_{n \to +\infty} \frac{1}{n^d} \sum_{j\in \bbZ^d} | \psi(j/n)- \psi( (j+e_i)/n)| =0$ 
for any $i=1,..,d$;
\item[(iii)] the limit $c(\psi):= \lim_{n \to +\infty} \frac{1}{n^d} \sum_{j\in \bbZ^d}  \psi(j/n)$ exists and is finite.
\end{itemize}
Then,  for any 
   measurable function   $f:\O \to \bbR$  in $L^p$ for some $\rot{p\in [1,+\infty)}$, it holds 
 \be\label{ciak25}
 \lim _{n\to +\infty}\frac{1}{ n^d} \sum_{j\in \bbZ^d} \psi( j/n) f( T^j \o) =c(\psi) \bbE[f\,| \,\cI]
 \en
 both $\bbP$--a.s. and in $L^p$.
\end{Theorem}\label{kittel}
\begin{Remark}If  $\psi$ is Riemann integrable, then Item (iii) above holds with  $c(\psi):= \int _{\bbR^d} \psi(x) dx$.
Moreover, the function $\psi(x):= (1+|x|)^{-\b}$ with $\b>2d+2$ \rot{fulfills} all  the  
assumptions of the above ergodic theorem.
\end{Remark}

We  introduce the shorthand notation
\be\label{scorciatoia}
W^\psi_n( f)(\o):=\frac{1}{n^{d}} \sum_{j\in \bbZ^d} \psi( j/n) f( T^j \o)\,.
\en
We point out that, whenever $|\psi(x)|\leq \vartheta(|x|)$ for a $d$--good function $\vartheta$, \rosso{then the series  $\frac{1}{n^d} \sum_{j\in \bbZ^d}  \psi(j/n)$ in Item (iii) of Theorem \ref{erg_thm} is absolutely convergent and therefore well defined. If in addition 
 $f\in L^p\subset L^1$, then   the series defining  $W^\psi_n (f)(\o)$ is absolutely convergent a.s.
 Indeed,
 the measure-preserving property of $T^j$  implies that
$\bbE[W^{|\psi|}_n (|f|)]= 
 \bbE[|f|] \frac{1}{n^{d}} \sum_{j\in \bbZ^d} |\psi( j/n)|<+\infty
$.}

\section{Applications to random measures on $\bbR^d$}\label{appl_random_meas}

 \rot{Differently from Section \ref{sec_erg_thm}, in this section $|x|$ will denote   the \ciak{Euclidean} norm of $x$. Moreover, given a topological space $W$, $\cB(W)$ will denote the $\s$--algebra of Borel subsets of $W$}.
 
  Let $\bbG$  be the abelian group $\bbR^d$ or $\bbZ^d$,  endowed with the standard Euclidean topology and the discrete topology, respectively.
 We suppose that $\bbG$ acts on the probability space $(\O, \cF, \cP)$. We call $(\theta_g)_{g\in \bbG}$  this action. This means that   the maps $\theta_g:\O\to\O$ satisfy the following properties:   $ \theta_0=\mathds{1}$;   $ \theta _g \circ \theta _{g'}= \theta_{g+g'}$ for all $g,g'\in \bbG$;  the map $\bbG\times \O \ni (g,\o) \mapsto \theta _g \o \in \O$ is measurable.

A set $A\in \cF$ is called $\bbG$--invariant if $A=\theta_g A$ for all $g\in \bbG$. 

\medskip

\noindent
{\bf Assumption $\mathbf{1}$}: \emph{We assume that 
 $\cP$ is $\bbG$-stationary, i.e.~$\cP\circ \theta_g^{-1}=\cP$  for all $g\in \bbG$. We also assume that  $\cP$ is ergodic, i.e. $\cP(A)\in\{0,1\}$ for any $\bbG$-invariant set $A\in \cF$.}

 \medskip

We fix a proper action  $(\t_g )_{g\in \bbG}$ of $\bbG$ on $\bbR^d$  given  by  translations. 
More precisely,  for a given invertible $d\times d$ matrix $V$, we have 
\be\label{trasferta1}
\t_g x = x + V g\,,  \qquad \forall x \in \bbR^d\,,\; \forall g \in \bbG\,.\en
In several applications $V=\mathbb{I}$, thus implying that   $\t_g x = x+g$. The case $V\not=\mathbb{I}$ is particularly relevant when treating e.g. crystal lattices \rot{\cite{F_SEP,Fhom1}}. \rot{An example with $V\not=\mathbb{I}$ is given in Section \ref{miele}.}

We denote by $\cM$ the metric space of locally finite non-negative measures on $\bbR^d$ with $\s$--algebra of measurable sets given by the Borel $\s$-algebra $\cB(\bbR^d)$  \cite[Appendix~A2.6]{DV}. \rot{The definition of the metric $d_\cM$ on $\cM$ is rather involved and is given in \cite[Eq.~(A2.6.1)]{DV}. We will not use the explicit expression of $d_\cM$. We just} recall that $\nu_n\to\nu$ in $\cM$ if and only if $\int_{\bbR^d} f(x)d\nu_n(x) \to \int_{\bbR^d} f(x)d\nu(x)$ for each real continuous function $f$  with compact support (shortly $f\in C_c (\bbR^d)$).  
The action of $\bbG$ on $\bbR^d$ naturally induces an action of $\bbG$ on  $\cM$, which (with some abuse of notation) we still denote by $(\t_g)_{g\in \bbG}$. In particular,  $\t_g : \cM\to \cM$ is given by 
 $\t_g \mathfrak{m} (A):= \mathfrak{m} (\t_{g} A)$ for all $ A\in \cB(\bbR^d)$ and  it holds
 \be\label{pollice}\int_{\bbR^d} f(x) d (\t_g \mathfrak{m})(x) =\int_{\bbR^d} f(\t_{-g}x) d \mathfrak{m}(x)\,.
 \en
\subsection{$\bbG$--stationary random measure $\mu_\o$ and rescaled random measure $\mu_\o^\e$}\label{cividale}
We suppose now to have a  random   {locally finite non-negative} measure $\mu_\o$  {on $\bbR^d$}, i.e. a measurable  map $\O \ni \o \mapsto \mu_\o \in \cM$. 
 The fundamental relation between the above two actions of $\bbG$ and  the random measure $\mu_\o$ is given by the following assumption:

 \smallskip 

\noindent
{\bf Assumption $\mathbf{2}$}: \emph{The random measure $\mu_\o$ is  $\bbG$--stationary: for all $\o \in \O$ and for all  $g \in \bbG$ it holds 
$ \mu_{\theta_g\o}= \t_g \mu_\o$. }

\smallskip

Calling $v^1,v^2,..., v^d$ the columns of  $V$,  we  introduce the parallelepiped
\be\label{birra}
\D:=\Big \{ \sum_{i=1}^d t_i v^i\,:\, 0 \leq t_i <1\Big\} \,.
\en
When $\bbG=\bbR^d$ one can also take   for $\D$ any bounded  Borel subset of $\bbR^d$ with finite and positive Lebesgue measure (e.g. $\D=[0,1)^d$).

\begin{Definition}\label{def_m}
\rot{The \emph{intensity} $m$ of the random measure $\mu_\o$ is defined as   $m:=\ell(\D)^{-1} \int_\O d\cP(\o) \mu_\o(\D)$, where $\ell(\D)$ is the  Lebesgue measure  of  $\D$.}
\end{Definition}

By the $\bbG$--stationarity of $\cP$, if $\bbG=\bbR^d$ then $\int_\O d\cP(\o) \mu_\o(U)=m \ell(U)$ for any \rosso{bounded} Borel set $U\subset \bbR^d$, while if $\bbG=\bbZ^d$ then $\int_\O d\cP(\o) \mu_\o(U)=m \ell(U)$ for any \rosso{bounded} set $U$ which is a union of sets of the form $\t_g \D$ with $g\in \bbZ^d$.

 \medskip



We introduce the rescaled measure $\mu^\e_\o$ defined as  
\be \label{scaletta}
\mu^\e_\o(A) = \e^d \mu_\o( \e^{-1} A) \qquad \forall A\in \cB(\bbR^d)\,.
\en
Note that it holds
\be\label{pioli} \int_{\bbR^d}  f(x)d \mu^\e_\o(x)=  \e^d \int_{\bbR^d}  f (\e x)d\mu_\o(x)
\en for any Borel function $f:\bbR^d\to\bbR_+$.

\begin{Definition}\label{def_goloso} \rot{We denote by $C_*(\bbR^d)$ the set of functions 
 $f\in C (\bbR^d) $ for which,  given  any $\b >0$, there exists 
    $ C>0$  such that   $|f(x)|\leq C (1+|x|)^{-\b}$  for all $x \in \bbR^d$. }
\end{Definition}

  We can now state our  limit theorem for $ \mu^\e_\o$, where convergence is stronger than the one in  $\cM$ itself (see Section \ref{settimino} for the proof):
\begin{Theorem}\label{teo_tartina} Suppose Assumptions $1$ and  $2$ to be valid  and  that the intensity $m$ is finite. Then
there exists a $\bbG$--invariant set $\cA\subset \O$ with $\cP(\cA)=1$ and with the following property.
Let $\varphi :\bbR^d\to \bbR$ be a continuous function such that, for some $C>0$ and  $\b>2d+2$,  $|\varphi (x)| \leq C (1+ |x|)^{-\b}$ for all $x\in \bbR^d$. Then for all $\o\in \cA$
the integral $\int_{\bbR^d}| \varphi ( x)| d\mu^\e_\o(x)$ is finite and 
it holds
\be\label{piazza_fiume}
\lim_{\e\da 0 }\int_{\bbR^d} \varphi ( x) d\mu^\e_\o(x) = m \int _{\bbR^d} \varphi(x) dx\,.
\en
In particular, \eqref{piazza_fiume} holds for all $\o \in \cA$ and all \rot{$\varphi\in C_*(\bbR^d)$}.
\end{Theorem}




The proof of Theorem \ref{teo_tartina} will use the following technical lemma, which will be important also for our applications to stochastic homogenization \rosso{and hydrodynamics} (see Section \ref{settimino} for the proof):
\begin{Lemma}\label{cinguetto} Suppose Assumptions $1$ and  $2$ to be valid  and  that the intensity $m$ is finite.
Then
there exists a $\bbG$--invariant set $\cC\subset \O$ with $\cP(\cC)=1$ and with the following property.
Fixed $\b>2d+2$ set $\vartheta (r):=(1+r)^{-\b}$ for $r\geq 0$. 
Then 
  for all $\o\in \cC$  \rosso{we have  $\int_{\bbR^d} \vartheta (|x|)d \mu_\o^\e (x) <+\infty$  and} 
\be\label{tordo}
\lim_{\ell \uparrow+ \infty}\, \varlimsup_{\e \da 0} \int_{\{|x|\geq \ell\}} \vartheta (|x|)d \mu_\o^\e (x) =0\,.
\en
\end{Lemma}

When the random measure has  higher finite density moments, one can deal  with a larger class of  functions. Indeed, by means \cite[Prop.~5.3]{TS} we can derive the following result
 where $\D$ is the fundamental cell defined in \eqref{birra}
(see Section \ref{WTS} for the proof):

\begin{Theorem}\label{prop_pizza} 
  Suppose Assumptions $1$ and  $2$ to be valid. In addition, assume that
  \rrr{$\bbE[ \mu_\o(\D)^\a]<+\infty$} for some $\a>1$.
  Fix a measurable function  $\vartheta :\bbR_+\to\bbR_+$ such that $\vartheta$ is   non-increasing,  is   convex on $[a,+\infty)$  for some $a>0$ and satisfies $\int _0^\infty r^{d-1} \vartheta(r)dr <+\infty$.
   Then 
   there exists a $\bbG$--invariant set $\rosso{\cA_\vartheta}\subset \O$ with $\cP(\rosso{\cA_\vartheta})=1$  such that, for all $\o\in \rosso{\cA_\vartheta}$,
   the integral $\int_{\bbR^d}| \varphi ( x)| d\mu^\e_\o(x)$ is finite and
\verde{\eqref{piazza_fiume} holds}   
for all continuous functions  $\varphi :\bbR^d\to \bbR$  with 
$|\varphi(x)| \leq \vartheta(|x|)$.
\end{Theorem}

Trivially, if $\bbE[ \mu_\o(\D)^\a]<+\infty$  for some $\a>1$, then the intensity $m$ if finite. \rot{The assumptions in Theorem \ref{prop_pizza} on $\vartheta$ are the same assumptions required in \cite[Prop.~5.3]{TS} and we have kept  them in their original form. An important example for applications is given by $\vartheta (r)=(1+r)^{-\b}$ with $\b>d$}.
The proof of the above theorem relies on the following lemma (proved in Section \ref{WTS}), relevant also for  our  applications to stochastic  homogenization and hydrodynamics:
\begin{Lemma}\label{cinguetto_bis}
  Suppose Assumptions $1$ and  $2$ to be valid. In addition, assume that  $\rrr{\bbE[ \mu_\o(\D)^\a]<+\infty}$ for some $\a>1$. \rosso{Fix a function $\vartheta$ as in Theorem \ref{prop_pizza}}.
  Then
  there exists a $\bbG$--invariant set $\rosso{\cC_\vartheta}\subset \O$ with $\cP(\rosso{\cC_\vartheta})=1$
\rosso{and such that, for all $\o\in \rosso{\cC_\vartheta}$,  it holds  $\int_{\bbR^d} \vartheta (|x|)d \mu_\o^\e (x) <+\infty$  and \eqref{tordo} is verified.}
\end{Lemma}

\rosso{By applying Theorem \ref{prop_pizza} and Lemma \ref{cinguetto_bis} to the \rot{countable} family of functions $\vartheta:\bbR_+\to\bbR_+$ of the form $\vartheta(r):=C(1+r)^{-\b}$ with rational $C,\b$ and $\b>d$, one gets the following immediate consequence:}
\begin{Corollary}\label{lollo_bronzo}\rosso{Suppose Assumptions $1$ and  $2$ to be valid. In addition, assume that  $\rrr{\bbE[ \mu_\o(\D)^\a]<+\infty}$ for some $\a>1$. Then both Theorem~\ref{teo_tartina} and \ciak{Lemma~\ref{cinguetto}} remain true if one substitutes the condition $\b>2d+2$ by the condition $\b>d$.}
  \end{Corollary}


For later use we point out that, as the reader can easily check from the proofs,  Theorems \ref{teo_tartina} and \ref{prop_pizza}, Lemmas \ref{cinguetto} and \ref{cinguetto_bis} and \rot{Corollary \ref{lollo_bronzo}} remain true if  Assumption $2$ is replaced by the following one:

\bigskip

\noindent
\rot{\bf Assumption $\mathbf{2^*}$}: \emph{It holds  $\mu_{\theta_g\o}= \t_g \mu_\o$  for all $g\in \bbG$ and all $\o$ varying in a $\bbG$--invariant set $\O_*\in \cF$ with $\cP(\O_*)=1$. }

\smallskip


\subsection{\rot{Examples}} In this section we provide some examples of random measures to which one can apply the above results. The class is very large and we just give some illustrative examples.

\subsubsection{Random measures associated to the Bernoulli bond percolation on $\bbZ^d$}
Consider the Bernoulli bond percolation on $\bbZ^d$. Denoting by $\bbE_d$ the set of undirected edges of the lattice $\bbZ^d$, the probability space is given by 
 $\O:=\{0,1\}^{\bbE_d}$ endowed with the product topology ($\{0,1\}$ has the discrete topology), $\cF$ is the $\s$--algebra of Borel sets and $\bbP$ is the Bernoulli product measure on $\O$ of parameter $p\in [0,1]$.
 Below we write $\o_{x,y}$ instead of $\o_{\{x,y\}}$ for $\{x,y\}\in \bbE_d$. 
 The group $\bbG:=\bbZ^d$ acts on $\O$ by the maps $\theta _g :\O\to \O$ with $g\in \bbG$, where   $(\theta_g \o)_{x,y}:= \o_{x+g,y+g}$ for all $\{x,y\} \in \bbE_d$.
 Note that Assumption $1$ is satisfied by $\bbP$.
 
 We consider the  translations $(\t_g)_{g\in \bbG}$ given by \eqref{trasferta1} with $V=\mathbb{I}$, i.e.~$\t_g x=x+g$ for all $x\in \bbR^d$ and $g\in \bbG$. Then the parallelepiped in \eqref{birra} is given by $\D=[0,1)^d$.
 
 An example of random measure satisfying Assumption  $2$ is $\mu_\o:= \sum _{x\in V(\o)} \d_x$, where $V(\o)$ is the vertex set of the graph obtained by keeping only the open edges, i.e. $V(\o):=\{ x\in \bbZ^d\,:\, \o_{x,y}=1 \text{ for some $y$  with $|x-y|=1$}\}$.

If $p$ is supercritical, then it is known that  there exists a $\bbG$-invariant set $ \O_1\subset \O$ such that the graph obtained by keeping only the open edges has a unique infinite connected component, and we call $\cC(\o)$ the vertex set of this component for $\o\in \O_1$. Then another  example of random measure satisfying Assumption  $2$ is given by  \[
\mu_\o:=
\begin{cases}  \sum _{x\in \cC(\o)} \d_x  & \text{ if } \o \in \O_1\,,\\
\emptyset &\text{ if } \o \not \in \O_1\,.
\end{cases}
\]
For both the above random measures  $\mu_\o(\D)\in \{0,1\}$, hence all moments of $\mu_\o(\D)$ are finite.  The first moment, i.e. the intensity $m$, if then given by $\bbP(0\in \cV(\o))$ and $\bbP(0\in \cC(\o))$ in the first and the second case, respectively.

\subsubsection{Contrast structures}\label{miele}
Let 
$\O:=\bbR_+^{\bbZ^2}$ endowed with the $\s$--algebra  $\cF$ of Borel subsets. Let $\bbP$ be  a probability measure on $\O$ satisfying Assumption $1$ where $(\theta_g \o)_z:= \o _{z+g}$ for any $g,z \in \bbZ^2$ (for example, $\bbP$ can be a  product probability measure  with equal marginals).

\begin{figure}
  \includegraphics[scale=0.17]{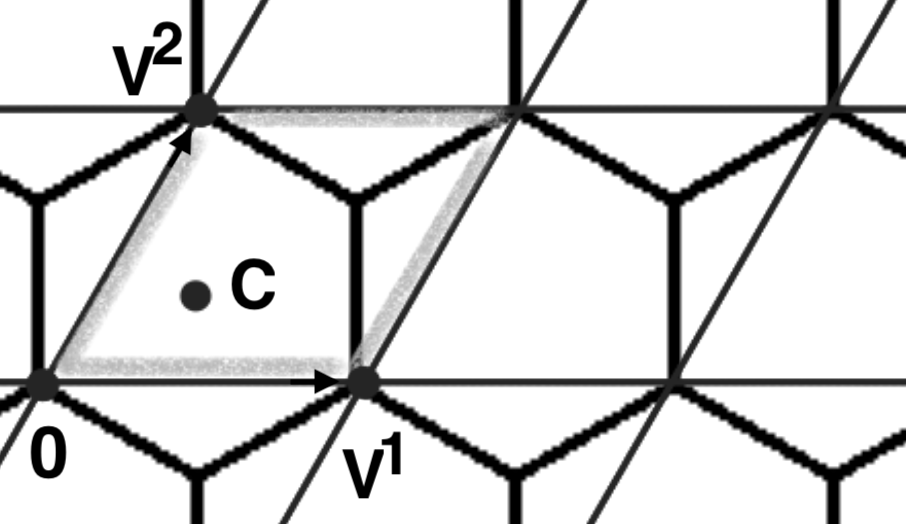} 
\caption{Hexagonal lattice, vectors $v^1$ and $v^2$, hexagon center $c$. The fundamental cell $\D$ is the parallelepiped with vertexes $0,v^1,v^2,v^1+v^2$ (apart boundary terms). }
\label{fig1}
\end{figure}

Consider the hexagonal lattice in $\bbR^2$. Let $v^1,v^2$ be the basis vectors and 
  $c$ be the center of the hexagon containing the origin as in Figure \ref{fig1}. Take $\D$ as in \eqref{birra}, set $V:=[v^1|v^2]$  and let $\t_g x:= x+ V g= x+ g_1 v^1 +g_2 v^2$ for all $g=(g_1,g_2)\in \bbZ^2$ and $x\in \bbR^2$. Note that the hexagonal lattice is left invariant by  the translations $(\t_g)_{g\in \bbZ^2}$. Moreover,  note that the map $g \mapsto  \t_g c$ is a bijection between $\bbZ^2$ and the set of hexagon centers. In what follows, given $g\in \bbZ^2$, we define $H_g$ as the hexagon centered at $\t_g c$. 

We set $\mu_\o (dx):=f_\o(x) dx$ where $f_\o(x):=\o_g$ for $x\in H_g$.
 Let us  check that the random measure $\mu_\o$  satisfies Assumption $2$. Since the density function $f_\o(x)$ is constant on the lattice hexagons, it is enough to check that $\mu_{\theta_g\o} (H_z)= \t_g\mu_\o (H_z)$ for all $\o\in \O$, $g,z\in \bbZ^2$.
We first observe that, since $H_z$ has center $\t_z c$, then $\t_g H_z$ has center $\t_g (\t_z c)=\t_{g+z} c$. This implies that $\t_g H_z= H_{g+z}$.  
  By denoting $A$ the area of a hexagon we then have 
 \[
 \t_g \mu_\o (H_z)=\mu_\o (\t_g H_z) =\mu_\o(H_{g+z}) = A \o _{g+z} = A(\theta _g \o)_z= \mu_{\theta_g \o}(H_z)\,. 
  \]
Hence, we have checked Assumption  $2$.

  Since  $\mu_\o(\D)= A ( \frac{4}{6}\o _0+\frac{1}{6} \o_{-v_2}+\frac{1}{6} \o_{-v_1})$,   the intensity $m$ is finite if and only if $\int d\bbP(\o) \o_0<+\infty$ and in general  $\bbE[\mu_\o (\D)^\a]<+\infty $ if and only if $\int d \bbP(\o) \o_0^\a<+\infty$.

    Suppose for example that $\o_0$ under $\bbP$ is a continuous random variable with probability density $\g \mathds{1}_{[e,+\infty)} (x)  (x \ln x)^{-2}$, where $\g$ is the normalizing constant.
    Since $1/[x ( \ln x)^2]=- D(1/\ln x)$,  $m$ is finite but $\bbE[\mu_\o (\D)^\a]=+\infty $ for all $\a>	1$ (in particular, one can apply e.g. Theorem \ref{teo_tartina} but not Theorem \ref{prop_pizza}).

\subsubsection{Random measures associated to  simple  point processes}
Let  $\O$   be the space of locally finite subsets of $\bbR^d$. The injective map $\Phi: \O\ni \o \mapsto \sum_{x\in \o}\d_x \in \cM$ allows to identify $\O$ with a subset of $\cM$ (indeed  $\Phi(\O)$ is a Borel subset of $\cM$ as stated in \cite[Prop.~7.1.III]{DV}). As in \cite{DV} we endow $\O$ with the metric induced by  $d_\cM$ and the above injection $\Phi$, i.e. $d(\o,\o'):= d_{\cM} (\Phi(\o),\Phi(\o'))$. 
We define $\cF$ as  the $\s$--algebra of Borel subsets of $\O$ w.r.t. the above metric. It is known (see~\cite[Corollary~7.1.VI]{DV})  that $\cF$ is generated by the sets $\{\o\in \O\,:\, |\o \cap A|=n\}$ with $A$ Borel set of $\bbR^d$ and $n $ in $\bbN$. Then $\bbG=\bbR^d$ acts on $\O$ by the maps $\theta_g\o := \o-g$ (cf.~\cite[Chapter~10]{DV}).  

Let  $\bbP$ be a probability measure on $(\O,\cF)$ satisfying Assumption 1, i.e. $\bbP$ is the law of any simple point process stationary and ergodic w.r.t. Euclidean translations according to the definitions of \cite{DV}  (e.g.~$\bbP$ is the law of a homogeneous Poisson point process).

We take $V=\bbI$ in \eqref{trasferta1}, i.e. we set $\t_g x:=x+g$ for all $x,g\in \bbR^d=\bbG$. 
Given $\o\in \O$, we set $\mu_\o := \Phi(\o)=\sum _{x\in \o } \d_x$. Then it is simple to check that also Assumption 2 is satisfied. Indeed, since  $\t_g \mu_\o (A)= \mu_\o (\t_g A)=\mu_\o (A+g)$, we have 
$\t_g \mu_\o= \sum_i \d_{x_i-g}$ if $\mu_\o =\sum_i \d_{x_i}$. This implies that  $\t_g \mu_\o=\mu_{\theta_g \o}$.

\section{Application to stochastic homogenization of long-range random walks on point processes and to hydrodynamics}\label{sec_rw}

 \rot{Differently from Section \ref{sec_erg_thm}, in this section $|x|$ will denote   the \rot{Euclidean} norm of $x$. We also recall that, given a topological space $W$,  $\cB(W)$ denotes its Borel $\s$--algebra. }

In this section we explain how to extend \cite[Theorem~4.4]{Fhom1} by removing  the restrictive Assumption (A9) present there. In \rot{order} to state our final result \rot{and give a self-contained presentation (accessible also without reading \cite{Fhom1}), we  recall} some results of \cite{Fhom1}.
We consider  the same setting of the previous section, \rot{but we restrict to purely atomic measures as described below.}
Let us recap our \ciak{setting.}

 
 \smallskip
 
 \rot{We have a probability space $(\O,\cF,\bbP)$. When modeling a disordered medium,  elements of $\O$ are usually called \emph{environments} and encode all the local randomness of the medium. The group $\bbG=\bbR^d$ or $\bbG=\bbZ^d$ acts on the probability space by the action $(\theta_g)_{g\in \bbG}$ and $\bbP$ is supposed to be $\bbG$-invariant and ergodic for this action (Assumption  $1$ of Section \ref{appl_random_meas}). We consider the translations $(\t_g)_{g\in \bbG}$ on $\bbR^d$ given in \eqref{trasferta1} where $V$ is an invertible matrix.  
We also  assume  to have a random measure $\mu_\o$ which   is   purely atomic (i.e.~pure point)  with locally finite support  for any $\o\in \O$. In particular, we have 
\be\label{cibo}
\mu_\o = \sum _{x\in \hat \o} n_x(\o) \d_{x}\,, \;\;\; \; n_x(\o) :=\mu_\o(\{x\})\,,\;\;\;\;\hat \o:=\{x \in \bbR^d\,:\, n_x(\o) >0\}
\en
 and $\hat \o$ is  a locally finite  set (we note that the map $\o\mapsto \hat \o$ then defines a simple point process \rot{according to \cite{DV})}. 
Finally, we assume that   $\mu_{\theta_g\o}= \t_g \mu_\o$  for all $g\in \bbG$ and all $\o$ varying in a suitable  $\bbG$--invariant set $\O_*\in \cF$ with $\cP(\O_*)=1$  (Assumption $2^*$ of Section \ref{appl_random_meas}).}

\smallskip

We \rot{enrich the above setting by  assuming} to have a measurable function 
\be r: \O\times \bbR^d \times \bbR^d \ni (\o, x, y)\mapsto r_{x,y}(\o) \in [0,+\infty)\,.\en 
As it will be clear below, only the value of $ r_{x,y}(\o)$ with $x\not =y$ in  $ \hat \o$ will be relevant.  Hence, 
without loss of generality, we take  
\be \label{lino}
r_{x,x}(\o)\equiv 0\,, \qquad  r_{x,y}(\o)\equiv 0 \qquad \forall  \{x,y\}\not \subset \hat \o\,.
\en
Below $r_{x,y}(\o)$, with $x,y\in \hat \o $, will \rosso{be the   jump rates of a continuous time random walk in the environment $\o$ with state space $\hat\o$}. Before introducing this random walk we fix some notation and assumptions.

\medskip
\rot{We recall that, roughly, the intensity $m$ is the mean density of points in $\hat \o$ (cf.~Definition \ref{def_m}).} 
Since we need to deal with the Palm distribution, we need that the intensity $m$ is \rosso{finite and non zero} (if $m$ was zero,  $\mu_\o$ would be the zero measure $\bbP$--a.s.).  Hence we introduce the following:
\smallskip

\noindent
{\bf Assumption $\mathbf{3}$}: \emph{The intensity $m$ is finite and positive.}

\smallskip

We call $\cP_0$ the Palm distribution associated to $\cP$ and the random measure $\mu_\o$.  \rot{Below} we recall  the definition of $\bbP_0$ \rot{(the interested reader can see  \cite[Section~2.3]{Fhom1} and references therein for a detailed exposition and proofs, although not necessary below)}. 

For  $\bbG=\bbR^d$ and $V=\bbI$
the Palm distribution $\cP_0$ is the probability measure on $(\O,\cF)$ such that, for any $U\in \cB(\bbR^d)$  with \rot{Lebesgue measure $\ell (U)\in(0,+\infty)$}, 
\be\label{palm_classica}
\cP_0(A):=\frac{1}{m  \ell(U) }\int _\O d\cP(\o) \int _{U} d\mu_\o (x)  \mathds{1}_A(\theta_x \o)\,, \qquad \forall A\in \cF
\,.\en
\ciak{The probability measure} $\cP_0$ has support inside the set $ \O_0:=\{\o\in \O\,:\, n_0(\o)>0\}$. We refer to \cite[Section~2.3]{Fhom1} for the case $\bbG=\bbR^d$ and $V$ generic \rot{(this case is not common, hence we do not detail it here)}.

For  $\bbG=\bbZ^d$,  $V=\bbI$, and $\hat \o\subset \bbZ^d$ for all $\o\in \O$ (this case will be  called \emph{special discrete case}),  the Palm distribution $\cP_0$ can be identified with   the  probability measure 
concentrated on the set $\O_0:=\{ \o \in \O:n_0(\o)>0\}$ such that
\be\label{palmeto}
 \cP_0(A):= \bbE\left[ n_0\,   \mathds{1}_A\right] / \bbE[n_0] \qquad \forall A\in \cF\,.
 \en
 In this particular case the intensity of the random measure is given by 
$ m= \bbE[n_0]$.

In general, for  $\bbG=\bbZ^d$,
the  Palm distribution $\cP_0$  is the probability measure on $\left(\O\times \D,\cF\otimes \cB(\D)\right)$ \rot{(where $\D$ is the parallelepiped \eqref{birra})} such that
\be\label{Palm_Z}
\cP_0(A):=\frac{1}{m\, \ell(\D)}\int _\O d\cP(\o) \int _{\D} d\mu_\o (x)  \mathds{1}_A(  \o,x )\,, \qquad \forall A\in \cF\otimes \cB(\D)
\,.\en
\ciak{The probability measure} $\cP_0$ has support inside $\O_0:=\{(\o, x)\in \O\times \D\,:\,n_x(\o)>0\}$. 
 \rot{If, in addition, $V=\bbI$ and $\hat \o \subset \bbZ^d$, then $\D=[0,1)^d$ and   $\O_0:=\{(\o, 0)\,: \o \in \O\,, \,n_0(\o)>0\}$. Hence, in the special discrete case,  by applying the natural bijection $(\o,0)\mapsto \o$ between  $\O_0$ and  $\{\o \in \O\,:\, n_0(\o)>0\}$, $\bbP_0$ defined in \eqref{Palm_Z} becomes \eqref{palmeto}. }


We define the function 
$\l_k:\O_0 \to [0,+\infty]$  (for $k\in[0,\infty)$) as follows:\\
 \begin{equation}\label{altino15}
 \begin{split}
 & 
 \begin{cases}
  \l_k(\o):=\sum _{x\in \hat \o}  r_{0,x}(\o)|x|^k\\
 \O_0=\{\o\in \O\,:\, n_0(\o)>0\}
 \end{cases}  \qquad
\Large{\substack{ \text{\;\;\;\;Case $\bbG=\bbR^d$ and}\\\text{\;\;\;\;\;\;\;\;special discrete case}}\,,}
\\
& \begin{cases}
  \l_k(\o,a):=
\sum_{x\in \hat \o}  r_{a,x}(\o)|x-a|^k   \\
   \O_0:=\{(\o, x)\in \O\times \D\,:\,n_x(\o)>0\}
\end{cases}
 \text{\;Case $\bbG=\bbZ^d$\,.}
\end{split}
\end{equation}

%

\medskip

\smallskip
\noindent
{\bf Assumption $\mathbf{4}$}. \emph{We assume that for some $\bbG$--invariant set  $\O_*\in \cF$  with $\cP(\O_*)=1$ the following conditions are fulfilled:
\begin{itemize}
\item[(i)]   for all $\o \in \O_*$  and  $g\not =g' $ in $ \bbG $, it holds 
  $ \theta_g\o\not = \theta _{g'} \o$;
\item[(ii)]   for all $\o \in \O_*$, 
  $g\in \bbG$ and  $x,y \in \bbR^d$, it holds  
$ r_{x,y} (\theta_g\o)= r_{\t_g x, \t_g y} (\o) $;
\item[(iii)]   for all $\o \in \O_*$  and    $x,y\in  {\hat \o}$,  it holds 
$n_x(\o)  r_{x,y}(\o) = n_y(\o) r_{y,x}(\o)$;
\item[(iv)]   for all $\o \in \O_*$ and  $x\not=y$ in  $  \hat \o$, there exists a  path $x=x_0$, $x_1$,$ \dots, x_{n-1}, x_n =y$ such that $x_i \in \hat \o$ and  $r_{x_i, x_{i+1}}(\o) >0$ for all $i=0,1, \dots, n-1$;
\item[(v)]   $ \l_0, \l_2 \in L^1(\cP_0)$;
  \item[(vi)]  $L^2(\cP_0)$ is separable.
\end{itemize}
}
\rot{Trivially, at cost to take the intersection, the sets $\O_*$ appearing in Assumption $2^*$ and $4$ can be considered  the same}.

\rot{We point out that Assumptions $1$, $2^*$, $3$, $ 4$  correspond to Assumptions (A1),\dots,(A8) in \cite[Section~2]{Fhom1} (indeed  (A1) is Assumption $1$, (A2) is  Assumption $3$, (A3) is Item (i) of Assumption $4$, (A4) is  Assumption $2$ plus Item (ii) in Assumption $4$, while (A5), (A6), (A7) and (A8) correspond respectively to Items (iii), (iv), (v) and (vi) of Assumption $4$). They are satisfied in plenty of models as discussed below. For  comments on the above assumptions we refer the interested reader to  \cite[Section 2.2]{Fhom1} and \cite{F_resistor}.}

\subsection{The random walk $(X^\o_t)_{t\geq 0}$ and stochastic homogenization}
Given $\o \in \O$, we consider the  continuous-time random walk $(X^\o_t)_{t\geq 0}$ with state space $\hat \o$ and jumping from $x$ to \rosso{$y\not =x $} with probability rate $r_{x,y}(\o)$.  In particular, once arrived at $x\in \hat\o$ the random walk waits there an exponential time with parameter $r_x(\o):=\sum_{y\in \hat \o} r_{x,y}(\o)$, afterwards it jumps to another site $y$ in $\hat\o$ chosen with probability $r_{x,y}(\o)/r_x(\o)$.
Due to \cite[Lemma~3.5]{Fhom1}, under general assumptions which are implied by our Assumptions $1$, $2^*$, $3$, $4$,  \rot{for all $\o$ varying in a $\bbG$--invariant set with $\bbP$--probability one  the above parameters $r_x(\o)$ are finite and positive for all $x\in \hat \o$, 
the random walk $(X^\o_t)_{t\geq 0}$ has a.s.  no explosion   (whatever the starting point) and therefore it  is well defined for all $t\geq 0$}.

We point out that  $(X^\o_t)_{t\geq 0}$ is a (possibly long-range) random walk on the simple point process $\hat \o$ with $\o$ sampled by $\bbP$. Our \rot{modeling} covers several examples \rot{(see Figure \ref{grafene}). In \cite[Section~5]{Fhom1} the reader can find the discussion (also about the validity of Assumptions $1$, $2^*$, $3$, $4$) of the following models: nearest-neighbor random conductance model on $\bbZ^d$, random conductance model on $\bbZ^d$ with long conductances, random walk with random conductances on infinite clusters, Mott random walk (whose underlying graph is the complete graph on $\hat\o$), 
simple random walk on the $d$-dimensional  Delaunay triangulation \rot{(i.e.~the graph dual to the Voronoi tessellation)} on a Poisson point process, 
nearest--neighbor random conductance models on lattices. The reader can find the discussion of the validity of Assumptions $1$, $2^*$, $3$, $4$  also in \cite[Section~3.4]{F_resistor} for stochastic lattices and periodic models, and in \cite[Section~5]{F_SEP} for random conductance models on crystal lattices and  for random walks on marked simple point precesses.
 Finally other examples will be provided in \cite{F_in_prep} and \cite{FT_in_prep}.}

%

\begin{figure}
  \includegraphics[scale=0.30]{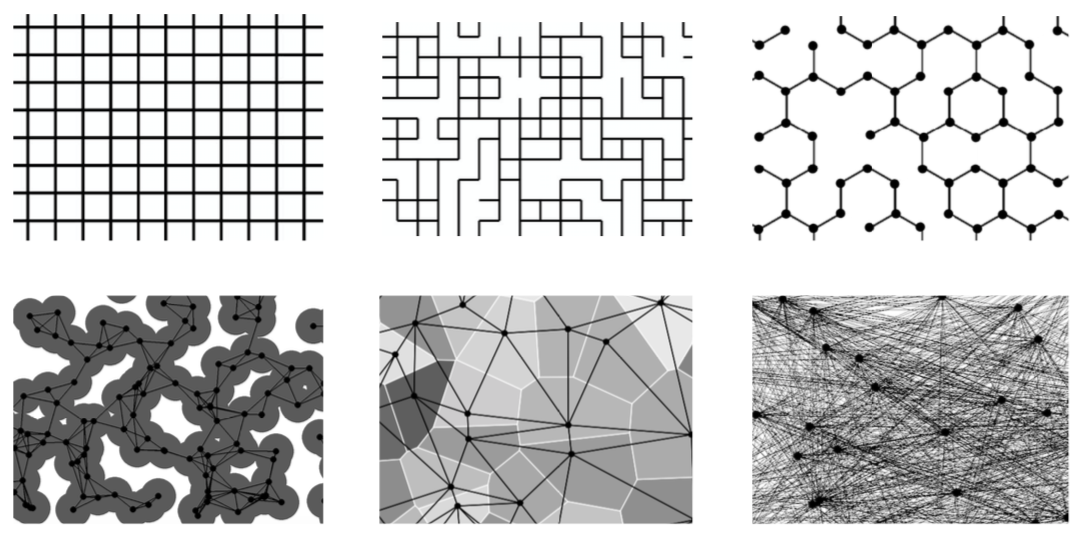} 
\caption{Some  random  graphs with vertex set $\hat \o$ underlying the random walk $(X^\o_t)_{t\geq 0}$: lattice $\bbZ^d$; the supercritical  percolation cluster on $\bbZ^d$, on the hexagonal lattice and  in the  Boolean model; the graph dual to the Voronoi tessellation; a  complete graph.}
\label{grafene}
\end{figure}

\medskip

Let $D$ be the \emph{effective homogenized matrix}.
$D$ is implicitly defined  as solution of a variational problem, moreover $D$ is  a symmetric $d\times d$  positive semidefinite matrix.  When $V=\mathbb{I}$, for   $\bbG=\bbR^d$ or in  the special discrete case, it holds 
\[
 a \cdot Da =\inf _{ f\in L^\infty(\cP_0) } \frac{1}{2}\int _{\O_0} d\cP_0(\o)\int_{\bbR^d} d\hat \o (x) r_{0,x}(\o) \left
 (a\cdot x - \nabla f (\o, x) 
\right)^2\,,
 \]
 for any $a\in \bbR^d$,   where $\nabla f (\o, x) := f(\theta_{x} \o) - f(\o)$.
 \rosso{We refer to \cite[Definition 3.6]{Fhom1} for the general case.}
 \rrr{We point out that $D$ can be degenerate and non-zero (see~\cite[Appendix~A]{F_SEP} for an example). }

\medskip

Given $\e>0$ we write $(P^\e _{\o ,t} )_{t \geq 0}$ for the $L^2(\mu^\e _\o)$--Markov semigroup associated to the \rot{diffusively rescaled} random walk $(\e X^\o _{ \e^{-2} t} )_{t\geq 0}$ on $\e \hat \o$, where 
\[\rot{\mu^\e_\o := \e^d \sum _{x\in \hat \o} n_x(\o) \d_{\e x}}\]
 according to \eqref{scaletta} and \eqref{cibo}. Simply, given $f\in
L^2(\mu^\e _\o)$,   $P^\e _{\o ,t}f(\rot{\e x})$ is the expectation of $f ( \e X^\o _{ \e^{-2} t})$ when the \rot{diffusively rescaled}  random walk starts at $\e x$.
We denote  by $ \bbL^\e _\o$ the infinitesimal generator of the semigroup 
$(P^\e _{\o ,t} )_{t \geq 0}$, which is a self-adjoint operator in  $L^2(\mu^\e_\o)$ (see \cite{Fhom1} for more details on $ \bbL^\e _\o$, \rot{although not used below}). 
Given $\l>0$ we write  $R^\e _{\o ,\l}: L^2(\mu^\e _\o)\to L^2(\mu^\e _\o) $ for the $\l$--resolvent  associated to the random walk $\e X^\o _{ \e^{-2} t} $, i.e.  $R^\e _{\o ,\l}:= (\l -\bbL^\e _\o )^{-1} =\int_0^\infty e^{- \l s} P^\e _{\o ,s} ds $.


Similarly we write $( P_t  )_{t \geq 0} $ for the  Markov semigroup on $L^2( m dx)$   associated to the  (possibly degenerate)  Brownian motion on $\bbR^d$  with diffusion matrix $ 2 D $. We denote by $\nabla_*$ the weak gradient along the space ${\rm Ker}(D)^\perp$ (when $D$ is non-degenerate $\nabla_*$ reduces to the standard weak gradient).
We write $ R_\l : L^2(m dx) \to L^2(m dx)$ for the $\l$--resolvent associated to the above Brownian motion on $\bbR^d$  with diffusion matrix $2 D$. \rot{We point out that, by Claim~\ref{sale78} in Section \ref{settimino},  for $\bbP$--a.a.~$\o$ the measure $m dx$ is the  limit   as $\e\da0$ of $\mu^\e_\o$  in the measure space $\cM$ (cf.~Section \ref{appl_random_meas}).}

\medskip

We recall a classical definition in stochastic homogenization:
\begin{Definition}\label{debole_forte} Fix $\o\in \O$ and a 
 family of $\e$--parametrized  functions $v_\e \in L^2( \mu^\e_\o)$. The family $\{v_\e\}$  \emph{converges weakly} to the function  $v\in L^2( m dx)$ (shortly, $v_\e \rightharpoonup v$) if $ \varlimsup  _{\e \da 0} \| v_\e\| _{L^2(\mu^\e_\o) } <+\infty$ and 
   $\lim _{\e \da 0} \int_{\bbR^d} d  \mu^\e _\o (x)   v_\e (x) \varphi (x)= 
\int_{\bbR^d} dx\,m  v(x) \varphi(x) 
$
for all $\varphi \in C_c(\bbR^d)$. 
The  family $\{v_\e\}$  \emph{converges strongly} to  $v\in L^2( m dx)$ (shortly, $v_\e\to v$)  if  $ \varlimsup  _{\e \da 0} \| v_\e\| _{L^2(\mu^\e_\o) } <+\infty$ and 
\[
\lim _{\e \da 0} \int_{\bbR^d}  d  \mu^\e _\o (x)   v_\e (x) g_\e (x)= 
\int_{\bbR^d} dx\, m  v(x) g(x) 
\]
for any family of functions $g_\e \in  L^2(\mu^\e_\o)$ weakly converging to $g\in L^2( m dx) $.
\end{Definition}
\begin{Remark}\label{rinomato}
  It is well known, and also recalled in Claim \ref{sale78} in \rot{Section \ref{settimino}}, that   there exists a $\bbG$--invariant set $\cB\in \cF$ with $\cP(\cB)=1$ such that \eqref{piazza_fiume} is true for all functions in $C_c(\bbR^d)$. As a consequence, for $\o \in \cB$, strong convergence implies weak convergence.
\end{Remark}

As stated in  \cite[Theorem~4.1]{Fhom1}, under Assumptions $1$, $2^*$, $3$, $4$ 
there exists a $\bbG$--invariant set $\O_{\rm typ}\in \cF$ with $\cP(\O_{\rm typ})=1$ such that 
for  all $\o\in \O_{\rm typ}$ and all  $\l>0$   
the massive Poisson equation \be\label{eq1} -\bbL^\e_\o u_\e+\l u_\e =f_\e\en
with $f_\e\in L^2(\mu^\e_\o)$ 
 stochastically \rot{homogenizes} towards the effective homogenized equation
  \be\label{eq2}  - \nabla_* \cdot D\nabla_* u + \l u =f \en
 with $f\in L^2(m dx)$ when $f_\e$ converges (weakly or strongly) to $f$. 
 \rot{The above homogenization corresponds to  the  convergence of solutions, i.e. 
  \be\label{fretta} f_\e \toup f \;\Longrightarrow \; u_\e \toup u\,, \qquad \qquad 
 f_\e \to f  \;\Longrightarrow \; u_\e \to u\,,
\en
 and the convergence of flows and energies (we refer to    \cite[Theorem~4.1]{Fhom1}   for a precise statement of flow and energy convergence, which anyway is  not used below).}
  Note that \eqref{eq1} and \eqref{eq2} can be rewritten as     $u_\e= R^\e _{\o ,\l}  f_\e$  and $u = R_\l f$, respectively. 
 
 

 

\smallskip

\rot{
In \cite[Section~4]{Fhom1} we introduced the additional assumption (A9) that we recall here:}

\rot{
{\bf Assumption (A9) in \cite{Fhom1}}: \emph{At least one of the following properties  is fulfilled:
 \begin{itemize}
 \item[(i)] For $\cP$--a.a.~$\o$ 
  $\exists C(\o)>0$ such that $\mu_\o ( \t_k  \D) \leq C(\o)$ for all $k\in \bbZ^d$.
  \item[(ii)] Setting $N_k(\o):=\mu_\o (\t_k\D)$ for $k\in \bbZ^d$, for some $C_0\geq 0$  it holds     $\bbE[N_0^2]<\infty$ and 
   \be\label{intinto}
| \text{Cov}\,(N_k, N_{k'})| \leq  
  C_0 |k-k'| ^{-1} 
   \en
  for any $k \not = k'$ in $\bbZ^d$. More generally,  we assume that, at cost to enlarge the probability space, one can define random  variables $(N_k) _{k\in \bbZ^d}$ with $\mu_\o (\t_k \D)\leq N_k$,  such that $\bbE[N_k],\bbE[N_k^2]$ are bounded uniformly in $k$ and  such that \eqref{intinto} holds for  all $k\not=k'$.
 \end{itemize}
 }}

\rot{Assumption (A9) was introduced to obtain in     \cite[Theorem~4.4]{Fhom1}  the limits  \eqref{marvel1}, \eqref{ondinoA}, \eqref{marvel2} and \eqref{ondinoB} of Theorem \ref{teo2} below   for $f\in C_c(\bbR^d)$.  We point out that  \eqref{marvel2} and \eqref{ondinoB} are crucial in getting the  hydrodynamic limit of interacting particle systems (see~\cite{F_SEP,F_in_prep}) and are derived from  \eqref{marvel1} and \eqref{ondinoA}. The main result  of this section, given by Theorem \ref{teo2} below,  is that Assumption (A9) is indeed unnecessary and can be removed thanks to our ergodic theorems. Moreover, the above limits are extended to continuous functions $f$ on $\bbR^d$ decaying suitably fast at infinity.}
 To state \rot{this} theorem we set
  \[
    \cG(r):=\big\{f \in  C (\bbR^d) \,:\,
    \exists C>0,\;\rot{\exists} \b>r  \text{ with }|f(x)|\leq C (1+|x|)^{-\b} \; \forall x \in \bbR^d \big\}\,.
    \]
    \rot{Note that the functional set  $C_*(\bbR^d)$ introduced in Definition \ref{def_goloso} equals $\cap _{r>0}\cG(r)$.}
Moreover, \rot{recall}  the fundamental cell $\D$ introduced in \eqref{birra}.

  \begin{Theorem}\label{teo2}  Let Assumptions $1$, $2^*$, $3$, $4$  be satisfied. Then there exists a measurable set $\O^*_{\rm typ}\subset\O_{\rm typ}$,
    $\bbG$--invariant and with  $\bbP$--probability one,
such that  as $\e \da 0 $ the following limits  hold for any $\o \in \O^*_{\rm typ}$,
   $t\geq 0$,  $\l>0$ and $f\in \rot{C_* (\bbR^d)}$: 
  \begin{align}
    &  L^2(\mu^\e _\o ) \ni P^\e_{\o,t} f \to P_t f \in L^2(mdx)\,, \label{marvel0}
    \\
    &   L^2(\mu^\e _\o ) \ni R^\e_{\o,\l} f \to R_\l  f \in L^2(mdx)\,, \label{vinello}\\
    &  \int \bigl | P^\e_{\o,t} f(x) - P_t f (x) \bigr|^2 d \mu^\e_\o (x)\to 0\,,\label{marvel1}\\
& \int \bigl | R^\e_{\o,\l} f(x) - R_\l f (x) \bigr|^2 d \mu^\e_\o (x)\to 0\,,\label{ondinoA}\\
    &  \int \bigl | P^\e_{\o,t} f(x) - P_t f (x) \bigr| d \mu^\e_\o (x)\to 0\,,\label{marvel2}\\ 
&    \int \bigl | R^\e_{\o,\l} f(x) -R_\l f (x) \bigr| d \mu^\e_\o (x)\to 0\,.\label{ondinoB}
  \end{align}
  More specifically,  we have
  \begin{itemize}
   \item[(i)]  \eqref{marvel0},  \eqref{vinello}, \eqref{marvel1} and \eqref{ondinoA} hold if   $f\in \cG(d+1)$;   \eqref{marvel2} and \eqref{ondinoB} hold if 
      $f\in  \cG(2d+2)$;
      \item[(ii)]
if in addition $\bbE\big[\mu_\o(\D)^\a]<+\infty$ for some $\a>1$,   \rot{then} \eqref{marvel0},  \eqref{vinello}, \eqref{marvel1} and \eqref{ondinoA} hold \rot{when}  $f\in \cG(d/2)$ \rot{while} \eqref{marvel2} and \eqref{ondinoB} hold \rot{when}
$f\in  \cG(d)$.
\end{itemize}
\end{Theorem}
The proof of Theorem \ref{teo2} is given in Section \ref{squillano}.

 \subsection{Hydrodynamics} 
 We conclude by presenting an application of our ergodic Theorem \ref{erg_thm} to the hydrodynamic limit of interacting particle systems.
 As discussed  in   \cite{F_SEP} and \cite{Fhom1}, the limits \eqref{marvel2} and \eqref{ondinoB} are fundamental tools to prove the quenched hydrodynamic behavior  of 
  multiple random walks on $\hat \o$  by adding a site exclusion or zero range interaction (combining stochastic homogenization and duality).   We refer also to \cite{CFS,F1,F2,F_SEP,F_in_prep,FJL} and references therein. 
  
In \cite{F_SEP} we have considered the same  setting presented above, with  
$n_x(\o)=1$ and $r_{x,y}(\o)=r_{y,x}(\o)$. In this case  the random walk $(X_t^\o)_{t\geq 0}$ on $\hat \o$ becomes  a random conductance model  on the simple point process $\hat\o$ \cite{Bi}. Let us  consider  the associated symmetric simple exclusion process (i.e.~ multiple random walks as above interacting by site exclusion). Roughly  its Markov generator is given by
\[
  \cL f (\eta):=\sum_{x\in \hat \o}\sum_{y\in \hat\o} r_{x,y}(\o) \eta(x)(1-\eta(y)) \left( f(\eta^{x,y})-f(\eta)\right)\,,\]
  where $\eta\in \{0,1\}^{\hat \o}$, $f$ varies in a suitable set of  real functions on $\{0,1\}^{\hat \o}$ (including   local functions) and $\eta^{x,y}$ is the configuration obtained by exchanging the occupation numbers $\eta(x)$ and $\eta(y)$. Then,
under Assumptions (A1),..,(A9) of \cite{Fhom1} \rot{(i.e.~our present Assumptions $1$,~$2^*$,~$3$,~$4$ plus (A9))} and the additional Assumption (SEP)\footnote{\ciak{Assumption (SEP)
states the following. Given the environment $\o$, consider  the random graph obtained from the  vertex set $\hat\o$  
by adding an  edge between distinct vertices    $x$ and $y$ in $\hat \o$ with probability $1- e^{-r_{x,y}(\o) t}$,  independently for each  pair $\{x,y\}$.  Then, $\bbP$--a.s.,  for $t>0$ small enough this random graph has a.s. only connected components of finite cardinality.}}
 \rot{introduced in  \cite{F_SEP}} (the latter used for the construction of the  process and the analysis of its Markov generator), we have derived the quenched hydrodynamic limit in path space  for the above symmetric simple exclusion process under  diffusive scaling.
 Assumption (A9) was used to get  \eqref{marvel2} and \eqref{ondinoB} from \cite{Fhom1} (see \cite[Proposition~6.1 \rot{and Remark~6.2}]{F_SEP}). Due to our  Theorem \ref{teo2} we then have the following consequence:
\begin{Corollary}\label{scherzetto}
The quenched hydrodynamic limit in path space stated in \cite[Theorem~4.1]{F_SEP} and described by the hydrodynamic equation $\partial_t \rho=\nabla\cdot(D \nabla \rho)$ holds without assuming Assumption (A9) there.
\end{Corollary}
In  \cite{F_in_prep}, \ciak{using the results we obtained in \cite{F_muratore},}  we will show that also Assumption (SEP) can be removed from  \cite[Theorem~4.1]{F_SEP}, hence we refer the interesting reader to \cite{F_in_prep} for a more detailed discussion.

\section{Proof of the Maximal Inequality (Theorem \ref{teo_max_in}) }\label{proof_max_in}

The proof is inspired by the one for the maximal inequality for the action of the group $\bbZ^d$  with averages on boxes \cite{K,S}, but we use a different covering procedure in order to control the effects of possible  non-zero tails of $\vartheta$. 

The following construction and  Lemma \ref{tasso} below are a standard tool to prove the maximal inequality. We recall them for completeness since crucial  below and since they are usually stated for boxes.
Fix  $I_1\subset I_2\subset \cdots \subset  I_N$    subsets of $\bbZ^d$.
Suppose to have a finite set $\cB\subset \bbZ^d$ and  a function $k: \cB\to\{1,...,N\}$. 
Define $\rot{\mathfrak{M}}_N$
as a  maximal  collection of points $z\in \cB$ such that $k(z)=N$  and the sets
 $z+I_N$ with $z\in \rot{\mathfrak{M}}_N$ are disjoint. 
Then define $\rot{\mathfrak{M}}_{N-1} $ 
as a  maximal  collection of points $z\in \cB$ such that $k(z)=N-1$  and the sets
$z+I_{N-1}$ with $z\in \rot{\mathfrak{M}}_{N-1}$ are reciprocally disjoint and disjoint from the sets $v+I_N$ with $v\in \rot{\mathfrak{M}}_N$. Proceed in this way until defining 
$\rot{\mathfrak{M}}_{1} $ as a  maximal  collection of points $z\in \cB$ such that $k(z)=1$  and the sets
$z+I_1$ with $z\in \rot{\mathfrak{M}}_1$ are reciprocally disjoint and disjoint from the sets $v+\rot{I_{k(v)}}$ with $v\in \rot{\mathfrak{M}}_2\cup \cdots \cup \rot{\mathfrak{M}}_N$.
Finally, we define
\[
     \cB':=   \rot{\mathfrak{M}}_1\cup \rot{\mathfrak{M}}_2\cup \cdots \cup \rot{\mathfrak{M}}_N\,.
   \]
   We stress that by construction all sets
$z+I_{k(z)}$, $k\in \cB'$, are disjoint.


 \verde{Below we use the standard notation $z+A-B:=\{z+a-b\,:\, a\in A\,,\; b\in B\,\}$.} 
\begin{Lemma}\label{tasso}
$\cB\subset \cup_{z\in \cB'} \left( z+ I_{k(z)}-I_{k(z)}\right)$ and  $|\cB|\leq \sum_{z\in \cB'} |I_{k(z)}-I_{k(z)}|$.
\end{Lemma}
The proof is \rot{similar to}  the proof of   \cite[Lemma~2.5, Section~6.2]{K}. \rot{We give it for completeness}.
\begin{proof}
Let us prove that $\cB\subset \cup_{z\in \cB'} \left( z+ I_{k(z)}-I_{k(z)}\right)$
(the conclusion then follows immediately).
Let $y\in \cB$. 
By the maximality of $\rot{\mathfrak{M}}_{k(y)}$, there are two possible cases: either  $y\in\rot{\mathfrak{M}}_{k(y)}$ or $y+I_{k(y)}$ intersects some $v+I_{k(v)}$ with  $k(v) \geq k(y)$ and $v\in \rot{\mathfrak{M}}_{k(v)}$ (hence $v\in \cB'$). In the first case, $y \in \cB'$ and trivially $y\in y+ I_{k(y)}-I_{k(y)}$. In the second case, \rot{there exist} $ a\in I_{k(y)}\subset I_{k(v)}$ and $b\in I_{k(v)}$ such that $y+a=v+b$, thus implying that $y=v+b-a\in v+I_{k(v)}-I_{k(v)}$.
 \end{proof}


\rot{We can start the proof of Theorem \ref{teo_max_in}. To this aim let} $\a>0$ and let $N \leq \ell$ be positive integers. 
 All constants of type  $c,C$  below have to be thought \rot{of as} finite, positive,   determined only  by $d,\vartheta,\rosso{\rho},\k$ (hence, independent from $\a,\rosso{\ell,N},f$) and  can change from line to line. 

\smallskip

 
 Recall \eqref{scorciatoia}. 
We  define
\begin{align*}
& E_N(\a):=\big\{\o\in \O: \sup_{1\leq n \leq N} W^\psi_n ( f) (\o) >\a \big\}\,,\\
&\cA(\o):= \{ z\in \bbZ^d : |z|\leq \ell\,,\; T^z \o \in E_N(\a) \}\,.
\end{align*}
To lighten the notation, sometimes dependence on the parameters   will be omitted (as in the definition of $\cA(\o)$, which indeed depends \rot{also} from $\a,\ell,N$).

\rot{The core of the proof will consist in proving   that, for some $C>0$, 
\be\label{rigore}
\a \bbE\big[ |\cA(\o)|\big] \leq C  \ell^d\|f\|_1 \,.
\en 
We first explain how to conclude once having \eqref{rigore}. Since each $T^z$ is measure-preserving, we have 
\be \label{novantesimo}
\begin{split}
\bbE\big[ |\cA(\o)|\big] &= \bbE[\sum _{z\in \bbZ^d: |z| \leq \ell} \mathds{1}\big( T^z\o  \in E_N (\a) \big) ]\\
&= \sum _{z\in \bbZ^d: |z| \leq \ell} \bbP\big(  T^z\o  \in E_N (\a) \big)\geq \grun{c\,\ell^d} \bbP\big( E_N(\a)\big)\,.
\end{split}
\en
By combining \eqref{rigore} with \eqref{novantesimo} we get 
$\bbP\big( E_N(\a)\big) \leq \grun{C'} \| f \|_1 /\a$ for each integer  $N \geq 1$ ($\ell$ has disappeared). To conclude the proof and get the maximal inequality \eqref{max_ineq}, it is enough to take the limit $N\to +\infty$ and use the  dominated convergence theorem. It remains now to prove \eqref{rigore}. }

\smallskip

Fix $z\in \cA(\o)$. We know that there exists $k(z)$ with $1\leq k(z) \leq N$ such that $W^\psi_{k(z)} (f) (T^z\o) > \a$ (in case of multiple possibile indexes $k(z)$, we take e.g. the minimal one). Hence we have 
\be\label{cruciale}
\rosso{\a k(z)^d <  k(z)^d W^\psi_{k(z)} (f) (T^z\o) =}\sum_{j\in \bbZ^d} \psi\Big( \frac{j-z}{k(z)}\Big) f( T^j \o) \qquad \forall z\in \cA(\o) \,.
\en
Let us consider the \verde{sets}
\[ \rosso{D(m,r):=\{x \in \bbZ^d\,:\, m r \leq |x|<(m+1)r\}\,, \qquad m \in \bbN\,,\; r=1,2,\dots,N\,.}\]
If $j\in z+D(m,k(z))$, then  $ m k(z) \leq |j-z|  < (m+1) k(z)$, thus implying that 
$  \psi\Big( \frac{j-z}{k(z)}\Big)  \leq \vartheta \Big(  \frac{|j-z|}{k(z)}\Big) \leq \vartheta (m) $
(recall that $\vartheta$ is \rosso{non-increasing}). In particular
from \eqref{cruciale} \verde{and since $\{z+D(m,r): m\in \bbN\}$ is a partition of $\bbZ^d$ for any $z\in \bbZ^d$ and $r=1,2,\dots,N$,}
 \rot{by taking $r:=k(z)$} we get
\be \label{luce}
\rosso{\a k(z)^d <} \sum_{m=0}^{\rosso{+\infty}} \sum_{j \in \bbZ^d} \vartheta(m) \mathds{1}_{ z+ D(m,k(z))}(j) f( T^j\o)  \qquad \rosso{\forall z\in \cA(\o)} \,.
\en
 
At cost to multiply the function $\rho$ by a positive constant, we can assume that $\sum_{m=0}^\infty\rho(m)=1$.  Hence, the  \rosso{l.h.s. of \eqref{luce}} can be rewritten as $\sum_{m=0}^{\rosso{+\infty}} \a k(z)^d \rho(m)$.
 As a consequence (recall that $f\geq 0$)  there must exist some $m(z) \in \bbN$
such that
\be \label{arcobaleno}
  \a k(z)^d \rho  \big(m(z)\big)<
\sum_{j \in \bbZ^d} \vartheta\big(m(z)\big) \mathds{1}_{D(z)}(j ) f( T^j\o)\,,
\en
where $ D(z):= z+ D\big(m(z),k(z)\big)$ 
(in case of many possible $m(z)$'s we take e.g. the minimal one).

\smallskip

For each  $m\in \bbN$  we set  $ \cA_m (\o):= \{ z \in \cA(\o)\,:\, m(z)=m \}$ and
 apply Lemma \ref{tasso} with $\cB:= \cA_m(\o)$, function $k: \cB \to \{1,2,\dots, N\}$ as above and sets $I_1^m\subset I_2^m\subset \cdots \subset I^m_N$ given by
\[
I^m_r:= \{x \in \bbZ^d\,:\, |x|< (m+1) r\} \qquad r =1,2,\dots,N\,. 
\]
The construction presented before Lemma \ref{tasso}  then produces a subset $\cA_m'(\o) $ of $\cA_m(\o)$ (i.e. $\cB'=\cA_m'(\o)$ is the set produced by the construction). By Lemma \ref{tasso} we then have 
\be \label{luna}
\begin{split}
| \cA_m(\o) | & \leq \sum_{z \in \cA_m'(\o)} \Big| I^m _{k(z)}-I^m_{k(z)}\Big| \leq c (m+1)^d\sum_{z \in \cA_m'(\o)}  k(z)^d\\ &= c\frac{ (m+1)^d}{\a \rosso{\rho} (m)}\sum_{z \in \cA_m'(\o)} \a  k(z)^d \rosso{\rho}(m)  \,.
\end{split}
\en
Above we used that  in general $  I^m_r-I^m_r\subset  \{x \in \bbZ^d\,:\, |x|< 2 (m+1) r\}$ \rot{and the last set}  has $c(d, \k) (m+1)^d r^d$ points \rot{(we recall that $|x|$ denotes the $\ell^\k$--norm of $x$ with   $k\in [1,+\infty]$)}.
As a byproduct of  \eqref{arcobaleno} and \eqref{luna}, and since $m(z)=m$ for all $z \in \cA'_m(\o) \subset \rosso{\cA_m(\o)}$, we  have 
\be \label{friuli10}
\a | \cA_m(\o) | <c\frac{ (m+1)^d}{ \rosso{\rho} (m)}  \sum_{z \in \cA_m'(\o)}  \sum_{j \in \bbZ^d} \vartheta(m) \mathds{1}_{D(z)} (j) f( T^j \o)\,.
\en

Recall that, given $z\in \cA'_m(\o)\rosso{\subset\cA(\o)}$, $|z| \leq \ell$ and  $k(z) \leq N \leq \ell$.  If $j \in D(z)$ and $z \in \cA_m'(\o)$, then $|j-z| <(m+1) k(z)$, thus implying that $|j|\leq |z|+(m+1) k(z) \leq (m+2) \ell$. In particular, in the r.h.s. of \eqref{friuli10} we can restrict to $j \in \bbZ^d$ with $|j| \leq (m+2) \ell$. Moreover, by construction the sets $z+ I^m_{k(z)}$ with $z \in \cA_m'(\o)$ are  disjoint, while $D(z) \subset z+ I^m_{k(z)}$ for all  $z \in \cA_m'(\o)$. Hence, also all the sets $D(z)$  with $z \in \cA_m'(\o)$  \rosso{are}  disjoint.  In particular, given $j \in \bbZ^d$ with $|j| \leq (m+2) \ell$, there is at most one $z\in \cA_m'(\o)$ such that $j \in D(z)$. These observations, the non-negativity of $f$ and \eqref{friuli10}  lead to 
\be\label{trieste25}
\a | \cA_m(\o) | <c\frac{ (m+1)^d}{ \rosso{\rho} (m)}\vartheta(m)   \sum_{\substack{j \in \bbZ^d:\\|j| \leq (m+2) \ell}}  f( T^j \o)\,.
\en
By taking the expectation of both sides of \eqref{trieste25} and  using  that each $T^j$ is measure-preserving we get
\be\label{biauzzo67}
\a \bbE\big[ |\cA_m(\o)|\big] \leq c \frac{ (m+2)^{2d}}{\rosso{ \rho} (m)}\vartheta(m) \ell^d \| f\|_1 \,.
\en  
Summing among all $m\in \bbN$ and using that $\cA(\o)$ is the disjoint union of the $\cA_m(\o)$'s, from \eqref{biauzzo67}
 \rot{we  finally get \eqref{rigore} with $ C:= c \sum_{m=0}^{+\infty}  \ (m+2)^{2d} \vartheta(m) \rho (m)^{-1} <+\infty$ (to get that $C$ is finite we used that $\vartheta$ is $d$--good, see    Definition \ref{goodness})}. 

\section{Proof of the Ergodic Theorem (Theorem  \ref{erg_thm})} \label{proof_erg_thm}
\rot{The main part of the  proof of Theorem  \ref{erg_thm} follows
from the Maximal Inequality of Theorem~\ref{teo_max_in} via a rather standard procedure (see e.g.~ the proof of \cite[Theorem~2.8, Section~6.2]{K} and  \cite[Theorem~2.6]{S}). We prefer to detail the arguments below, both to keep the exposition self-contained and in order to highlight where conditions (i), (ii) and (iii) are used.}

Recall our shorthand notation $W^\psi_n (f)(\o):=n^{-d} \sum_{j\in \bbZ^d} \psi( j/n) f( T^j \o)$.
%
 We call  $\tilde \cI\subset \cF$ the $\s$--subalgebra given by the almost-invariant sets, i.e.~$\tilde \cI:=\{ A\in \cF\,:\, \bbP((T_{\rosso{i}} ^{-1} A) \Delta A)=0 \text{ for all } 1\leq \rosso{i}\leq d\}$. Trivially, $\cI \subset \tilde \cI$.   It is known  that, given $A\in \cF$, $A \in \tilde \cI$ if and only if there exists $D \in \cI$ such that $\bbP(A \D D)=0$  (\rot{for $d=1$ see e.g.~}\cite[Exercise~\rosso{6.1.2}-(iii), Chapter~6]{Du} \rot{and for a generic $d$ take  $D:= \cap _{n\in \bbN} (\cup_{j\in\bbN_{ n} ^d} T^{-j } A )$ where $\bbN_n:=\{m\in \bbN\,:\, m\geq n\}$}).
  It then follows that $\bbE[f|\cI]= \bbE[f|\tilde \cI]$ $\bbP$--a.s.\,. Hence, we just need to prove Theorem  \ref{erg_thm} with  $\bbE[f|\tilde\cI]$ instead of  $\bbE[f|\cI]$.  In what follows,  we call a measurable function $g$ on $\O$ \emph{\rot{almost-}invariant}  if $g\circ T_i= g $ $\bbP$--a.s. for all $i=1,2,\dots, d$ (this is equivalent to the fact that $g$ is $\tilde \cI$--measurable).

\rot{As detailed  in the proof of  \cite[Theorem~2.8, Section~6.2]{K}}, given $\e>0$  one can write $f=\sum_{i=1}^d (g_i - g_i \circ T_i) +h +\varphi$, for suitable functions such that 
 $g_1,\dots, g_d\in L^\infty$, $h\in L^1 $ is 
  \grun{almost--}invariant for all $i=1,2,\dots, d$ \rot{and}  $\|\varphi\|_1<\e$.


Given  a \rot{measurable and integrable function  $u:\O\to \bbR$}, we define the function 
\[
\D (u) (\o) := \varlimsup _{n\to +\infty} W^\psi_n (u)(\o) -\varliminf_{n\to +\infty} W^\psi_n (u)(\o)\,.
\]
The above limsup and liminf are bounded in modulus by $\sup_{n\geq 1} W_n^{|\psi|}(|u|)(\o)$,
 which is finite $\bbP$--a.s. by the maximal inequality (applied with $|\psi|$ instead of $\psi$). As a consequence their difference, i.e. $\D (u)(\o)$,  is well-defined  and finite for $\bbP$-a.a.~$\o$ and in particular for all $\o \in \G(u)$, where 
 $\G(u):=\{\o\in \O\,:\, \sup _{n\geq 1} W^{|\psi|}_n(|u|)(\o)<+\infty\}$
 (note that $\G(u)\in \cF$ \rot{and $\bbP(\G(u))=1$}).

\smallskip 
$\bullet$ Our first target is to prove that  $\D(f)(\o)=0$ for $\bbP$--a.a.~$\o$, thus implying that  $\bar f(\o):= \lim _{n \to+\infty} W^\psi_n\rosso{( f)}(\o)$ is well defined and finite for $\bbP$--a.a. $\o$.


 It is simple to check that $\D$ is subadditive, i.e. $\D(f_1+f_2)(\o) \leq \D(f_1)(\o)+ \D(f_2)(\o)$ for all $\o \in \G(f_1)\cap \G(f_2)$. In particular, by writing our $f$ as $f=\sum_{i=1}^d (g_i - g_i \circ T_i) +h +\varphi$ as above, we have 
 \be\label{rientro}\D(f) (\o) \leq \sum_{i=1}^d \D(g_i - g_i \circ T_i)(\o) +\D(h)(\o) +\D(\varphi)(\o)
 \en
 for all 
 $ \o \in \G:= \G(f)\cap \left(\cap_{i=1}^d \G(g_i - g_i \circ T_i) \right)\rosso{ \cap \G(h)} \cap \G(\varphi)$. \rot{Note that $\bbP(\G)=1$}.

  We claim that  $\D(g_i - g_i \circ T_i)(\o)=0$  
    and $\D(h)(\o)=0$ for $\bbP$--a.a. $\o$ in $\G$.
Let us start with $g_i-g_i\circ T_i$. Since  $g_i$ is bounded and due to Item  (ii) in Theorem \ref{erg_thm}, 
we get 
\[
 \Big | W^\psi_n(g_i-g_i\circ T_i) (\o) \Big| 
 \leq \frac{\|g_i\|_\infty }{n^{d}} \sum_{j\in \bbZ^d} |\psi( j/n)- \psi((j-e_i)/n) |\stackrel{n\to \infty}{\rightarrow} 0\,.
\]
 As a consequence, 
 $\D(g_i - g_i \circ T_i)(\o)=0$  for all $\o\in \G$.
Let us move to $\D(h)(\o)$ for $\o\in\G$. 
The \rot{almost-}invariance of $h$ implies 
 that  $h(T^j \o)=h(\o)$ for all  $\o\in \G'$,   where $\G'$ is a measurable set with $\bbP(\G')=1$.  Due to Item  (iii) in Theorem \ref{erg_thm}, 
   we then  have 
  \be\label{limB}
  W^\psi_n(h) (\o)=\frac{h(\o)}{  n^{d}} \sum_{j\in \bbZ^d} \psi( j/n) \stackrel{n\to \infty}{\rightarrow} c(\psi) h(\o)  \qquad \forall \o \in \G\cap\G'
    \,.
    \en This implies that $\D(h)(\o)=0$ for all $\o\in \G\cap\G'$, thus concluding the proof of our claim.
    
    As a byproduct of \eqref{rientro} and the above claim, we conclude that $\D(f)(\o) \leq \D(\varphi)(\o)$ for $\bbP$--a.a. $\o\in \G$, \rot{while $\bbP(\G)=1$}. By this observation,  the maximal inequality (cf.~Theorem \ref{teo_max_in}) and since $\|\varphi\|_1<\e$, we have 
        \begin{equation*}
        \begin{split}
        & \bbP( \D(f) >\sqrt{\e})  \leq \bbP( \D(\varphi) >\sqrt{\e})        \\
        & \leq 
    \bbP\Big( 2 \sup _{n\geq 1} \frac{1}{n^{d}} \sum_{j\in \bbZ^d} |\psi|( j/n)\, |\varphi|( T^j \o)>\sqrt{\e}\Big)  
     \leq  \frac{2 C \|\varphi\|_1}{\sqrt{\e}}\leq  2 C \sqrt{\e}\,,
    \end{split}
    \end{equation*}
    for some $C=C(d, \vartheta, \rho, \k)$.
    By the arbitrariness of $\e>0$, we conclude that $\D(f)(\o)=0$ for $\bbP$--a.a. $\o$.
    
$\bullet$ We now prove  that $\bar f$ is \grun{almost--}invariant, i.e.~$\bar f=\bar f\circ T_i$  $\bbP$--a.s. for all $i=1,\dots, d$. 
By the measure-preserving property of $T_i$, the limit $ \lim _{n \to+\infty} W^\psi_n (f)(T_i \o)$ exists and is \rot{finite} $\bbP$--a.s. and according to our notation it is given by $\bar f (T_i\o)$. 
 By the  Markov inequality and the measure\rosso{-}preserving property of $T^j$  we get
\be\label{fornite}
\bbP(|W^\psi_n(f)(\o)-W^\psi_n(f) (T_i\o)|>\e)
 \leq \frac{\|f\|_1}{\e n^d}  \sum_{j\in \bbZ^d}\big| \psi(j/n)-\psi( (j-e_i)/n)\big| \,.
\en
Since, by Item (ii),   $\g(n):=n^{-d} \sum_{j\in \bbZ^d}\big| \psi(j/n)-\psi( (j-e_i)/n)\big| \to 0$, for each 
$k\in \bbN_+$  we can find $n_k\in \bbN_+$ such that $\g(n_k) \leq k^{-3}$ and such that the sequence $n_k$ is increasing. At this point, by taking $n=n_k$ and $\e= 1/k$ in \eqref{fornite} and afterwards  by applying  Borel-Cantelli lemma, we conclude that  $\lim_{k\to +\infty} \left[ W^\psi_{n_k}(f)(\o)-W^\psi_{n_k}(f) (T_i\o)\right]=0$ for $\bbP$--a.a.~$\o$. Since the limit equals $\bar f(\o)-\bar f (T_i\o)$  for $\bbP$--a.a.~$\o$, we get that $\bar f$ is \grun{almost--}invariant.


\smallskip

$\bullet$  We now  prove that $W_n^\psi ( f) \to \bar f$ in $L^p$ since $f\in L^p$. To this aim, we observe that for $f\in L^\infty \cap L^p$ this follows \rot{(as $p\in [1,+\infty)$)} from the above proved a.s. convergence $W_n^\psi( f) \to \bar f$ and the dominated convergence theorem  (for the latter  use that \rosso{$\|W_n^\psi (f)\|_\infty  \leq
\| f\|_\infty \frac{1}{n^{d}} \sum_{j\in \bbZ^d} |\psi( j/n)|$ and the r.h.s.} is bounded uniformly in $n$ \rosso{by Item (iii)}). To  
 extend the convergence to any $f\in L^p$,  
 we proceed as follows.
 First we point out that, 
   given $h\in L^p$, \rot{for some $C=C(\psi)$ it holds}
\be\label{tucano}
\| W^\psi_n( h)\|_p \leq \frac{1}{ n^d} \sum_{j\in \bbZ^d} |\psi( j/n)|\, \|  h\circ  T^j \|_p=\frac{\| h\|_p }{ n^d} \sum_{j\in \bbZ^d} |\psi( j/n)| \leq C \|h\|_p\,.
\en
Above, the  identity follows from the measure-preserving property of $T^j$ and the last bound follows from \rosso{Item (iii)}.  
Given $\d>0$ let $g\in  L^\infty\cap L^p$ with $\|f-g\|_p\leq \d$ ($g$ exists since $L^\infty\cap L^p$ is dense in $L^p$). By \eqref{tucano} applied to $h=f-g$ and  since $W_n^\psi(g)\to \bar g$ in $L^p$ we get that   $\| W^\psi_n( f)-\bar g \|_p\leq 2C \d$ for  $n$ large. By the arbitrariness of $\d$, this proves that $(W^\psi_n( f):n\geq 1)$ is a Cauchy sequence in $L^p$ and therefore it converges to some $\hat f\in L^p$. This implies that $W^\psi_n( f)\to \hat f $  $\bbP$--a.s. along a subsequence. Since $W^\psi_n( f)\to \bar f$ $\bbP$--a.s., we conclude that $\hat f=\bar f$ $\bbP$--a.s. and therefore   $W^\psi_n( f)\to \bar f$ in $L^p$.

\smallskip
$\bullet$ Finally we prove   that $\bar f = c(\psi) \bbE[ f|\tilde\cI] $ $\bbP$--a.s.
To this aim  set $a_n:=n^{-d} \sum_{j\in \bbZ^d}\psi(j/n)$ and 
take \rot{an almost--}invariant (i.e. $\tilde\cI$--measurable) bounded  function $g$. Then (using that $g\circ T^{-j}=g$ $\bbP$--a.s. and  $\bbE[g|\tilde\cI]=g$) we get 
\begin{equation*}
\begin{split} 
\bbE\left [ g \bbE[W^\psi_n\rosso{(f)}|\tilde\cI]\right]& =\bbE\left [  \bbE[gW^\psi_n\rosso{( f)}|\tilde\cI]\right]=
\bbE [ g W^\psi_n\rosso{( f)}] = \frac{1}{n^d} \sum_{j\in \bbZ^d}\psi(j/n)\bbE[ g (f\circ T^j)]\\
&=\frac{1}{n^d}\sum_{j\in \bbZ^d}\psi(j/n)\bbE[ (g\circ T^{-j}) f ]
= a_n \bbE[ g  f ]  =a_n \bbE[g \bbE[f|\tilde\cI]]\,.
\end{split}
\end{equation*}
 This prove that $\bbE[ W^\psi_n\rosso{( f)} | \tilde \cI]= a_n \bbE[f|\tilde\cI]$. 
 We know that 
  $W^\psi_n\rosso{( f)} \to \bar  f\;$ 
  in $L^p$. On the other hand, conditional expectation is a contraction in  $L^p$ (see \cite[Theorem~4.1.11, Chapter.~4]{Du}), thus implying that 
   $\bbE[ W_n^\psi\rosso{( f)} | \tilde\cI]\to \bbE[\bar f|\tilde\cI]$ in $L^p$. Since $\bar f$ is \grun{almost--}invariant we have $ \bbE[\bar f|\tilde\cI]=\bar f$. 
 By combining the above observations we have that $a_n \bbE[f|\tilde\cI]\to \bar f$ in $L^p$. On the other hand, $a_n \to c(\psi) $ by Item (iii) and this allows to conclude that $\bar f= c(\psi) \bbE[f|\tilde\cI]$.


\section{Applications to random measures: Proof of Theorem \ref{teo_tartina} and Lemma \ref{cinguetto}}\label{settimino}
In this section we first show how to derive Theorem \ref{teo_tartina} from Lemma \ref{cinguetto}, afterwards we prove the latter. In what follows, given $\ell>0$ we consider the ball  $B(\ell):=\{x\in \bbR^d\,:\, |x|\leq \ell\}$ and set  $B(\ell)^{\rm c}:=\bbR^d\setminus B(\ell)$.

\rot{All constants of type  $c,C$  below have to be considered finite, positive and  not dependent from $\o$ and the parameters $\e, \ell, n$ introduced in the rest. Moreover, their value can change from line to line.} 

\subsection{Proof of Theorem \ref{teo_tartina}}  \rot{The conclusion concerning $C_*(\O)$ is trivial and we focus on the rest}.
Trivially, it is enough to prove the stated property   for the family of  rational  constants  $C>0$ and $\b>2d+2$. By countability, we  just need to prove the statement for a fixed $C>0$ and a fixed $\b>2d+2$. Without loss of generality we can assume $C=1$. We  set $\vartheta (r):=(1+ r)^{-\b}$ for $r\geq 0$ \rosso{and take $\varphi\in C(\bbR^d)$ with $|\varphi(x) |\leq \vartheta(|x|)$ for all $x\in \bbR^d$.}

By Lemma \ref{cinguetto} the integral $\int_{\bbR^d} |\varphi(x) | d\mu_\o^\e(x)$ is finite for all $\o\in \cC$.  From now on we restrict to $\o\in \cC$.
Given a positive $\ell\in \bbN$ take  $\varphi_\ell\in C_c(\bbR^d)$ such that   $\varphi_\ell=\varphi$ on the ball $B(\ell)$ and $|\varphi_\ell (x)| \leq \vartheta (|x|)$ for all $x$.
Then, $|\varphi-\varphi_\ell |$ is zero on $B(\ell)$ and is bounded by $2 \vartheta(|\cdot|)$. This implies that 
\be\label{parrot1}
\Big  | \int_{\bbR^d} \varphi ( x)d\mu^\e_\o(x)- \int_{\bbR^d}  \varphi_\ell ( x)d\mu^\e_\o(x) \Big | \leq 2\int_{B(\ell)^c} \vartheta (| x|)d\mu^\e_\o(x)\en
and
\be\label{parrot2}
\Big |     \int _{\bbR^d} \varphi(u) du-
 \int _{\bbR^d} \varphi_\ell(u) du\Big | 
\leq 2 \int _{ B(\ell)^{\rm c}} \vartheta (|u|) du\,.\en
The last integral
  goes to zero as $\ell\to +\infty$ since $\vartheta (|\cdot|)$ is integrable on $\bbR^d$.

  \begin{Claim}\label{sale78}  There exists a $\bbG$--invariant set $\cB\in \cF$ with $\cP(\cB)=1$ such that \eqref{piazza_fiume} is true for all functions in $C_c(\bbR^d)$.
    \end{Claim} \rot{The above claim} is 
 usually stated without any proof. Since  the derivation is short, we give it for completeness in Appendix \ref{WT1}.

 Due to Claim \ref{sale78},
for all $\ell \in \bbN$  and $\o\in \cB$ we have 
$
\lim_{\e\da 0 } \int_{\bbR^d} \varphi_\ell ( x)d\mu^\e_\o(x) = m \int _{\bbR^d} \varphi_\ell(x) dx$.
By the above observation, \eqref{parrot1} and \eqref{parrot2}  the conclusion follows from Lemma \ref{cinguetto} and by taking $\cA:=\cB\cap \cC$ with $\cC$ as in Lemma \ref{cinguetto}.

\subsection{Proof of Lemma \ref{cinguetto}}   \rot{Since the integral in \eqref{tordo} is decreasing in $\ell$, \eqref{tordo} with $\ell\in \bbR_+$ is equivalent to \eqref{tordo} with $\ell\in \bbN$. From now on $\ell$ has to be considered in  $\bbN$.}
     We set $\psi(x):=\vartheta(|x|)$ and $\psi_\ell (x):=\vartheta(|x|)\mathds{1}(|x|\geq \ell)$.
  Then  \eqref{tordo} in Lemma \ref{cinguetto} \rot{can be restated as}
\be\label{duomo}
\lim_{\ell \uparrow \rot{+}\infty}\, \varlimsup_{\e \da 0} \int _{\bbR^d}  \psi_\ell(u)d\mu^\e_\o(u)=0\,.
\en
We call $\cC$ the set of $\o\in \O$ satisfying \eqref{duomo}, i.e.
\be \label{def_cC} \rot{\cC:=\big\{\o \in \O\,:\,\lim_{\ell \uparrow +\infty}  \varlimsup_{\e \da 0} \int _{\bbR^d}  \psi_\ell(u)d\mu^\e_\o(u)=0\big \}}\,.
\en
\rot{To prove Lemma \ref{cinguetto} it is enough  to show the following:   $\cC$ is measurable, $\cC$ is  $\bbG$-invariant, for all $\o \in \cC$  and $\e>0$ the integral $ \int _{\bbR^d}  \psi(x)d\mu^\e_\o(x)$ is finite,  $\bbP(\cC)=1$. 
The first three   properties correspond (with a different order)  to Claims \ref{claim_fin}, \ref{claim_meas} and  \ref{claim_inv}  below. 
We first show  here how the  ergodic Theorem \ref{erg_thm}   implies that    $\bbP(\cC)=1$. 
To this aim we let 
 $\D(z):=\t_z  \D$  (recall   \eqref{birra})  and
set $f( \o):= \mu_\o ( \D)=\mu_\o( \D(0))$. Then, by Assumption 2, we have
\be\label{ortica}
 \mu_\o( \D(z) ) =   \mu_\o( \t_z \D) = \mu_{\theta_z \o} (\D)=  f (\theta_z \o) \qquad \forall z\in \bbZ^d\subset \bbG\,.
 \en
 Note that $\int _{ \O}| f(\o)|d\bbP(\o) = m \ell (\D)<+\infty$, hence $f\in L^1 (\O)$.
 As stated in Claim \ref{claim_pieno} below,    \eqref{duomo} is satisfied whenever
 \be\label{anatra}
 \lim_{\ell\uparrow \infty}
 \varlimsup_{n\uparrow \infty}
 n^{-d} \sum_{z\in \bbZ^d}\psi _{\ell}(z/n)f(\theta_z \o)=0\,,
 \en
 where $\ell , n\in \bbN$. 
 Hence, to prove that $\bbP(\cC)=1$,  we just need to show that the measurable set of  the $\o$'s satisfying \eqref{anatra} has probability one.
}

\rot{Let $e_1,\dots, e_d$ be the canonical basis of $\bbZ^d$. For $i=1,\dots,d$ we  define $T_i:\O \to \O$ as $T_i =\theta_{e_i}$. Then $T_1,\dots, T_d$ are $d$ commuting  measure-preserving bijective maps on $(\O,\cF,\cP)$ (by Assumption $1$ and since $\bbZ^d\subset \bbG$). Moreover the map $T^g:= T_1 ^{g_1}\circ T_2^{g_2}\circ \cdots \circ T_d^{g_d}$ equals $\theta_g$ for all $g=(g_1,g_2,\dots, g_d) \in \bbZ^d$. Hence in \eqref{anatra} we can replace $\theta_z$ by $T^z$.
By Theorem \ref{erg_thm} with $\psi=\psi_\ell $ and with   $\vartheta(r)=(1+r)^{-\b}$, $\b>2d+2$,  there exists $\cC_*\in \cF$ with $\cP(\cC_*)=1$ such that 
for all $\o\in \cC_* $ the following holds: for each $\ell \in \bbN_+$ 
\be\label{iff89}
\lim_{n\uparrow  +\infty} n^{-d}  \sum_{z\in \bbZ^d}\psi _\ell (z/n)f(\theta_z \o)=\bbE[ f|\cI](\o)  \int_{\bbR^d} \psi_\ell (u) du\,,
\en
where 
$\cI:=\{A\in \cF\,:\, \theta _z A=A \;\; \forall z\in \bbZ^d\}
$ and $\bbE[ f|\cI](\o)$ is a fixed version of the conditional probability of $f$ w.r.t. $\cI$.
Since the integral in the r.h.s. of  \eqref{iff89} goes to zero as $\ell\to +\infty$, we conclude that  \eqref{anatra} holds for all $\o \in \cC_*$ and  therefore  $\bbP$--a.s.
This concludes the proof of Theorem \ref{erg_thm} once having Claims \ref{claim_fin}, \ref{claim_meas}, \ref{claim_inv}, \ref{claim_pieno}. 
The rest of the section is devoted to the above claims  and their proofs}.

\begin{Claim}\label{claim_fin} 
If $\o\in \cC$,  then  $ \int _{\bbR^d}  \psi(x)d\mu^\e_\o(x)$ is finite for all $\e>0$.
\end{Claim}
\begin{proof} By \eqref{pioli} the above integral \rot{$ \int _{\bbR^d}  \psi(x)d\mu^\e_\o(x)$} equals $\e^d \int _{\bbR^d} \psi_\ell (\e u) d\mu_\o(u)$. Hence for $\o\in \cC$  there exists $\ell$ such that $\e^d \int _{\bbR^d} \psi_\ell (\e u) d\mu_\o(u)\leq 1$ for $\e $ small. Since  $\mu_\o$ is locally finite, we then have that  $\int _{\bbR^d} \psi  (\e u) d\mu_\o(u)<+\infty$ for $\e$ small. \rot{This property extends to all $\e>0$ since
the integral  $\int _{\bbR^d} \psi  (\e u) d\mu_\o(u)$
 is non-increasing} in $\e$ as $\vartheta$ is non-increasing.
  \end{proof}


\begin{Claim}\label{claim_meas} \rot{The limit \eqref{duomo} holds  if and only if it holds with $\e$ varying in $\{1/n\,:\, n\in \bbN_+\}$}.  In particular,  $\cC$ is measurable.
\end{Claim}
\begin{proof}
\rot{We} observe that given $\e\leq 1$  and $n\in \bbN_+$ with
 $(n+1)^{-1} \leq \e \leq n^{-1}$,  we have (since $\vartheta$ is non-increasing \rosso{and due to \eqref{pioli}})
 \be
 \begin{split}& \int _{\bbR^d}  \psi_\ell (x)d\mu^\e_\o(x) =
 \e^d \int_{\{ \e |x|\geq \ell\}} \psi(\e x) dx
                                                          \leq n^{-d}  \int_{\{  |x|\geq \ell n \}} \psi( x/(n+1)) dx \\
   & \leq C (n+1)^{-d}  \int_{\{  |x|\geq \frac{\ell}{2}( n+1) \}} \psi( x/(n+1)) dx= C \int _{\bbR^d}  \psi_{\frac{\ell}{2}} (x)d\mu^{1/(n+1)}_\o(x)
 \end{split}
 \en
 and similarly
  \be
 \begin{split}& \int _{\bbR^d}  \psi_\ell (x)d\mu^\e_\o(x) \geq 
 (n+1)^{-d} \int_{\{  |x|\geq \ell (n+1)\}} \psi( x/n) dx
   \\
   &\geq 
              c   n^{-d}  \int_{\{  |x|\geq 2 \ell n \}} \psi( x/n) dx
  = c \int _{\bbR^d}  \psi_{2 \ell } (x)d\mu^{1/n}_\o(x)\,.
 \end{split}
 \en
 As a consequence, \eqref{duomo} is equivalent to the same expression but \rot{with} $\e\in \{1/n\,:\,n\in \bbN_+\}$. Dealing with a countable family of parameters $\e, \ell$ one \rot{obtains}  the measurability of $\cC$ \rot{since the map $\O\ni \o \mapsto  \int _{\bbR^d}  \psi_\ell (x)d\mu^\e_\o(x)\in \bbR$ is measurable}. 
\end{proof}

\begin{Claim}\label{claim_inv} $\cC$ is  $\bbG$-invariant.
\end{Claim}
\begin{proof}
  It is enough to prove that  \rot{for all} $ \o \in \cC$ and $g \in \bbG$ it holds $\theta _g \o\in \cC$. Indeed, this means that $\theta_g \cC\subset \cC$  \rot{for all} $ g\in \bbG$. By applying $\theta_g^{-1}=\theta_{-g}$ to both sides, we get $ \cC\subset\theta_{-g} \cC$ \rot{for all} $g\in \bbG$, i.e.  $ \cC\subset\theta_{g} \cC$ for all $g\in \bbG$. Hence $\theta_g \cC=\cC$  \rot{for all} $g\in \bbG$.
  
  Take $\o\in \cC$ and $g\in \bbG$. Due to \eqref{pollice} and \eqref{pioli} we can  write
 \be\label{giulia}
 \begin{split}
\int_{\bbR^d} \psi_\ell ( x)d\mu^\e_{\theta_g\o}(x)& =  \int_{\{|x|\geq \ell\}} \psi(x) d\mu^\e_{\theta_g\o}(x)  =\e^d \int_{\{|\e x|\geq \ell\}} \psi(\e x) d\mu_{\theta_g \o}(x)\\ & =\e^d \int_{\{|\e (x-V g)|\geq \ell\}} \psi(\e( x-V g) ) d\mu_{ \o}(x)\,.
 \end{split}
 \en
Since $g$ is fixed,  given $\ell>0$, for  $\e$ small we have $|\e V  g| \leq \ell/2$ and therefore 
the set $\{x:|\e (x- Vg)|\geq \ell\}$ is included in the set $\{x:|\e  x |\geq \ell/ 2 \}$. 
Moreover, for $\e$ small, we have  $|\e V  g| \leq \ell/4$ and therefore, given $x$ satisfying 
 $ |\e  x |\geq \ell/2 $, we have 
 \rot{$ | \e (x-Vg)| \geq | \e x| - | \e Vg| \geq |\e x| - \ell/4 \geq | (\e/2) x|$}. Since $\vartheta$ is non-increasing, the last bound implies that $\psi(   \e (\rot{x-Vg}) ) \leq \psi(  (\e/2) x)$. \rot{Hence, we can estimate the last integral in \eqref{giulia} by first integrating on $\{x:|\e  x |\geq \ell/ 2 \}$, then by replacing the integrand with $ \psi(  (\e/2) x)$ and afterwards by extending the integration to  $\{x:|\e  x |\geq \ell/ 4 \}$}.
 The above observations and \eqref{giulia} finally imply that
 \begin{equation*}
   \begin{split}
  &   \varlimsup_{\ell\uparrow+\infty}\varlimsup_{\e\da 0} \int_{\bbR^d} \psi_\ell ( x)d\mu^\e_{\theta_g\o}(x)
    \leq \\ &  \varlimsup_{\ell\uparrow+\infty}\varlimsup_{\e\da 0} 
              2^d (\e/2)^d \int_{\{|(\e/2) x |\geq \ell/4 \}} \psi((\e/2) x  ) d\mu_{ \o}(x)=\\&  2^d  \varlimsup_{\ell\uparrow+\infty}\varlimsup_{\e\da 0}   \int_{\{|x|\geq \ell/4\}} \psi(x) d\mu^{\e/2}_\o(x)    \,.
              \end{split}
 \end{equation*}
 To conclude it is enough to observe that the above r.h.s. is zero since $\o\in \cC$. We have therefore proven that  $\theta_g\o \in \cC$.
  \end{proof}

\begin{Claim}\label{claim_pieno}  \rot{If $\o$ satisfies \eqref{anatra}, then $\o \in \cC$}.
\end{Claim}
\begin{proof}
%
\rot{The proof is divided in two steps. First we 
  show that for some $c>0$ it holds
\be\label{campane}
  \int_{\bbR^d}  \psi_\ell (u)  d\mu^\e_\o(u)
  \leq c\, \e^d \sum_{z\in \bbZ^d}\psi _{\ell -1}(\e V z)f(\theta_z \o)\,.
 \en
 Afterwards we show how to  remove $V$ from \eqref{campane} getting that, for some $C,c>0$, it holds 
 \be\label{campane100}
  \int_{\bbR^d} \psi_\ell (u)  d\mu^\e_\o(u) 
  \leq C\, \e^d \sum_{z\in \bbZ^d}\psi _{\frac{\ell -1}{c}}(\e  z)f(\theta_z \o)\,.
 \en
 Since by Claim \ref{claim_meas} we can  vary  $\e$ among $\{1/n\,:\, n \in \bbN_+\}$ and   due to \eqref{campane100},  \eqref{duomo} is satisfied whenever \eqref{anatra} holds, thus proving our claim.
 }

$\bullet$ \rot{Proof of \eqref{campane}.}   By the form of $\vartheta$  there exists $c>0$ such that 
 \be\label{steppino}
 \vartheta(r) \leq  c   \,  \vartheta (s)  \, \qquad \forall r,s\geq 0 \text{ with } |s-r|\leq 1\,.
 \en
 From now on we restrict to \rosso{$\e \leq  {\rm diam}(\D)^{-1}$}, where ${\rm diam}(\D)$ denotes the  \ciak{Euclidean} diameter of $\D$.  
 Then, by   \eqref{steppino}, for $z\in \bbZ^d$ and  $x\in \D(z)$ we can bound 
 $\psi (\e x)\leq c\, \psi(\e V z )  $. Indeed,  since $\D(z)=\t_z \D$ and $0\in \D$, both $x$ and $\t_z 0=Vz$ belong to $\D(z)$ and therefore 
 \be\label{atreo} \big| \, |\e x|- | \e Vz| \,\big | \leq \e |x-Vz| \leq   \e \,{\rm diam}(\D)\leq 1\,.
 \en
The above bound and \eqref{steppino} allow to conclude that  $\psi (\e x)\leq c\, \psi(\e V z )  $ \rot{for all $x\in \D(z)$}.
By combining this result  with the identity $\mu_\o^\e( \e \D(z))= \e^d f(\theta_z\o)$ (which follows from \eqref{scaletta} and \eqref{ortica}),
we get
 \be\label{vino}
 \int _{\e \D(z)}  \psi(u) d \mu^\e_\o(u)  \leq c\,\e^d \psi(\e V z) f (\theta_z \o)  \,.
 \en
 By \eqref{atreo},  if $|\e x| \geq\ell$ and $x\in \D(z)$ then $|\e Vz | \geq   \ell-1 $.
 This observation and \eqref{vino} imply \rot{\eqref{campane}}.
 
%
 
 $\bullet$ \rot{Proof of \eqref{campane100}. }  
 Since $V$ is invertible, we have $c_1 |x| \leq |V x | \leq c_2 |x|$ for all    $ x\in \bbR^d$   for some constants $c_1,c_2 >0$. 
 Hence we have 
  \[
    \psi(\e V z)  \leq (1+\e c_1 |z|)^{-\b}\leq (\min\{1,c_1 \})^{-\b} (1+\e |z|)^{-\b}= (\min\{1,c_1\})^{-\b}  \rosso{\psi}    (\e z)\,,
  \]
  while $|\e V z|\geq \ell -1$ implies  $|\e z| \geq (\ell-1)/c_2$.
  Hence, from \eqref{campane}, we \rot{get \eqref{campane100}}.
\end{proof}


\section{Applications to random measures: Proof of Theorem   \ref{prop_pizza} and Lemma \ref{cinguetto_bis}}\label{WTS}

As in  the proof of Theorem \ref{teo_tartina}, Theorem \ref{prop_pizza} can be derived from  Lemma \ref{cinguetto_bis},  which we now focus on.
\rot{We set  $\psi(x):=\vartheta(|x|)$ and $\psi_\ell(x)=\vartheta(|x|)\mathds{1}(|x|\geq \ell)$. We define  $\cC=\cC_\vartheta$ as in \eqref{def_cC}.} 
Then  Claims \ref{claim_fin},  \ref{claim_meas} \rot{and \ref{claim_inv} of the previous section} are still valid, with the same proofs (we use there only that $\vartheta$ is non-increasing).



It remains to prove \rot{that $\bbP(\cC)=1$}.  To this aim we will use \cite[Proposition~5.3]{TS} \rot{when $\bbG=\bbR^d$ as detailed in Section \ref{polpette} below, thus completing the proof for $\bbG=\bbR^d$. We point out that \cite[Proposition~5.3]{TS}  is thought of  only  for the action of the group $\bbG=\bbR^d$ on $\O$ and cannot be applied directly to the case $\bbG=\bbZ^d$. To overcome this problem for $\bbG=\bbZ^d$, inspired by some methods known in homogenization theory,  in Section \ref{WT2} we build a   new probability space on which the group $\bbR^d$ acts  and a new random measure,  suitably related to the original ones. The results proved for the case $\bbG=\bbR^d$ and applied to this new context will imply that $\bbP(\cC)=1$ also for $\bbG=\bbZ^d$. }

\subsection{Proof that \rot{$\bbP(\cC)=1$} for  $\bbG=\bbR^d$} \label{polpette}
\rot{Before entering the technical details we present  the main part  of the proof. 
Since $\bbG=\bbR^d$ and  $\bbP$ is ergodic,   \cite[Proposition~5.3]{TS} with $t_n:=n$ implies the following. Consider a measurable function $h:\bbR_+\to \bbR_+$ such that $h$ is   non-increasing,  is   convex on $[a,+\infty)$  for some $a>0$ and satisfies $\int _0^\infty r^{d-1} h(r)dr <+\infty$. Then, 
 fixed a measurable function $f:\O\to \bbR$  with $f\in L^\a$ and $\a>1$,  for $\bbP$--a.a.~$\o$ it holds
\be\label{scarpetta}
 \lim _{n\to +\infty} n^{-d} \int_{\bbR^d} h(|x|/n) f(\theta_ x \o) dx = 
  \bbE[f]  \int _{\bbR^d} h (|x|) dx \,.
  \en
Note that the integral in the l.h.s. of \eqref{scarpetta} is w.r.t. the Lebesgue measure and not w.r.t. $\mu_\o$. To infer from \eqref{scarpetta} some  information  on $\mu_\o$ itself,  in part similarly to   the proof of \cite[Theorem~10.2.IV]{DV}, we apply \eqref{scarpetta} with  $f(\o) :=\int _{\bbR^d} \g_\d (x) d\mu_\o(x)$ and $h(r):=\vartheta (\min\{0, r-\d\})$, 
where $(\g_\d)_{\d>0}$ is a family of mollifiers described below. This will allow us to prove that, for some set $ \cC_1\in \cF $ with $\bbP( \cC_1)=1$, for all $\o \in  \cC_1$ it holds
\be\label{portabis}
 \varlimsup _{\e\da 0 }\e^d \int_{\bbR^d} \vartheta(\e |x|)  d \mu_\o (x) \leq
 m \int_{\bbR^d} \vartheta (|x|) dx\,.
 \en  
 } 
  
  \rot{We postpone the proof of \eqref{portabis}   and we first explain how to conclude. The strategy is to show that  $\cC_1\cap \cB\subset \cC$, where $\cB$ is   as in Claim \ref{sale78}. Since $\bbP( \cC_1)=\bbP(\cB)=1$, this implies  that $\bbP(\cC)=1$. To show that $\cC_1\cap \cB\subset \cC$, we proceed as follows. 
From Claim \ref{sale78} it is trivial\footnote{It is enough to  approximate from above and below  the map  $y\mapsto \mathds{1} (|y| < \ell)$ by suitable functions in $C_c(\bbR^d)$.} to get for any $\o \in \cB$  and any $\ell \in \bbN$ that 
\[ \lim _{\e\da 0 }\e^d \int_{\{ \e|x| < \ell\}} \vartheta(\e |x|)  d \mu_\o (x) =
m \int_{\{|x|<  \ell\}} \vartheta (|x|) dx\,.
\]
Subtracting the above limit  \ciak{from} \eqref{portabis} and then taking the limit $\ell\uparrow +\infty$, we get
\be\label{ratto}
\varlimsup_{\ell\uparrow +\infty} \varlimsup _{\e\da 0 }\e^d \int_{\{|\e x|\geq \ell\}} \vartheta(\e |x|)  d \mu_\o (x)\leq m \varlimsup_{\ell\uparrow +\infty}  \int_{\{|x|\geq \ell\}} \vartheta (|x|) dx =0
 \en  
for any $\o \in \cB \cap  \cC_1$. Note that the equality in \eqref{ratto} is due to   the integrability of $\vartheta(|x|)$, which follows from our assumptions. Due to  \eqref{pioli}, \eqref{ratto} leads to \eqref{duomo}, hence $\o\in \cC$.
This proves that $ \cC_1\cap \cB \subset \cC$.  }

\smallskip

\rot{At this point, we have just to prove \eqref{portabis}. In order to simplify the computations below, we first show} that  we can assume $V=\bbI$ without any loss of generality. To this aim we 
define $\tilde \theta _g:= \theta_{V^{-1} g}$ and $\tilde{\tau}_g:= \tau_{V^{-1} g}$. Note that $\tilde{\tau}_g x=x+V(V^{-1} g)=x+g=x+\bbI g$ and, by Assumption 2, 
\[\mu _{\tilde{\theta}_g \o}= \mu _{\theta_{V^{-1}g} \o}= \t _{ V^{-1} g } \mu_\o= \tilde{\t}_g \mu_\o\,.
\]
Since $\bbG=\bbR^d$ 
it is then trivial to check that  the actions $(\tilde{\theta}_g)_{g\in \bbG}$ and $(\tilde{\tau}_g)_{g\in \bbG}$, together with the random measure $\mu_\o$, satisfy Assumptions 1 and 2 and  that a set $A\in \cF$ is $\bbG$-invariant for the action $(\tilde{\theta}_g)_{g\in \bbG}$ if and only if the same holds for the action $({\theta}_g)_{g\in \bbG}$. \rot{In particular, if the claim in Lemma \ref{cinguetto_bis} is valid for the new setting with the new actions, then it is valid also for the original one}.  Then, at cost to pass to the actions 
$(\tilde{\theta}_g)_{g\in \bbG}$ and $(\tilde{\tau}_g)_{g\in \bbG}$, we can (and we do) assume that $V=\bbI$.

We fix a smooth function $\g:\bbR^d \to \bbR_+$ with $\g(x)$ determined by $|x|$, with  support in $B(1):=\{x\in \bbR^d\,:\, |x|=1\}$ and satisfying $\int _{\bbR^d} \g(x) dx =1$. Given $\d>0$ we consider the mollifier $\g_\d(x):=\rot{ \d^{-d}} \g(x/\d)$ (which is a probability kernel with support in $B(\d)$).

  Similarly to   the proof of \cite[\rosso{Theorem}~10.2.IV]{DV} we consider $f:\O\to [0,+\infty)$ defined as $f(\o) :=\int _{\bbR^d} \g_\d (x) d\mu_\o(x)$,  \rot{moreover we set $h(r):= \vartheta(\min\{0, r-\d\})$ (the dependence of $f$ and $h$  from $\d$ is omitted in the notation)}.
  Since   $\bbE[ \mu_\o ( B(\d)) ^\a]<+\infty$ (by our moment assumption on $\mu_\o(\rrr{\D})$)  and  $\g_\d$ is uniformly bounded, we get that $f\in \rot{L^\a}$.  In particular, \rot{by
 \eqref{scarpetta} and by varying $\d$ in the countable set $\{1/k:k\in \bbN_+\}$},    there exists a set $\cC_1\in \cF$ such that \rot{\eqref{scarpetta} is valid}  for all $\o\in \cC_1$ \rot{and for all $\d$ as above}.
  By stationarity $\bbE[f]= m \int _{\bbR^d} \g_\d (x) dx=m$.
On the other hand, by Assumption 2,  we can write (see  \eqref{pollice} \rot{and use that $\t_{-x}y=y-x$ as $V=\bbI$})
\[ f(\theta_x \o)= \int _{\bbR^d} \g_\d (y) d \mu_{\theta_x \o} (y)=  \int _{\bbR^d} \g_\d (y-x) d \mu_{ \o} (y)\,.
\]
Setting $z=x-y$ and using  also that $\g(z)=\g(-z)$ we get
\be\label{sole55}
\begin{split}
  \int_{\bbR^d} \rot{h}(|x|/n) f(\theta_ x \o) dx & = \int_{\bbR^d} dx\, \rot{h}(|x|/n)\int _{\bbR^d} \g_\d (x-y) d \mu_{ \o} (y)\\
& =                                                      
  \int_{\bbR^d} d \mu_\o (y)  \int_{B(\d)}  \rot{h}(|y+z|/n)  \g_\d(z) dz \,.
  \end{split}
\en
\rot{For $z\in B(\d)$ we have $-\d+|y+z|/n\leq |y|/n$ and since $\vartheta $ is non-increasing we get 
$h( |y+z|/n)\geq  \vartheta (|y|/n)$. Hence from \eqref{sole55} we get
\be\label{sole1974}
  \int_{\bbR^d} \rot{h}(|x|/n) f(\theta_ x \o) dx \geq   \int_{\bbR^d}  \vartheta(|y/n|)
d \mu_\o (y)   \,.
\en
}
By combining \eqref{scarpetta}  with \eqref{sole1974} and using that $\bbE[f]=m$, we get for all $\rot{\o\in \cC_1}$
that 
 \be\label{porta}
 \varlimsup _{n\to +\infty}n^{-d} \int_{\bbR^d}  \vartheta(|x|/n)d \mu_\o (x) \leq
 \rot{m \int_{\bbR^d} h(|x|) dx} \,.
 \en
\rot{Trivially the last integral
    converges to $\int_{\bbR^d} \vartheta(|x|) dx$ as $\d\da 0$. 
     Since moreover $\vartheta$ is non-increasing (given $\e\in (0,1)$ take $n\in \bbN_+$ with $(n+1)^{-1} \leq \e < n^{-1}$), from \eqref{porta}  we finally get \eqref{portabis} for all $\o\in \cC_1$.}

\subsection{Proof \rot{that $\bbP(\cC)=1$}   for  $\bbG=\bbZ^d$}\label{WT2} \rot{As already mentioned, we} move from the action of the group $\bbZ^d$ to the action of the group $\bbR^d$ by a standard method in homogenization theory. In particular, below  we \rot{will} apply \rot{some} results \rot{of \cite[Section~6]{Fhom1} that we will recall along the proof} (there we treated locally bounded  atomic random measures \rosso{on $\bbR^d$}, but the results remain valid for generic locally bounded random measures \rosso{on $\bbR^d$}). 
To simplify the presentation and the notation we consider only the case  $V=\bbI$ (the treatment in \cite{Fhom1} is for all $V$). \rosso{In this case $\D=[0,1)^d$.}


We set $\bar \Omega=\Omega\times[0,1)^d$ and call  $\cB$ the Borel $\sigma$--field of $[0,1)^d$. We consider the product $\s$-algebra  $\bar{\cF}=\cF\otimes\cB$ and   the product probability measure $\bar{\cP}=\cP\otimes dx$ on $\bar\Omega$. Then $(\bar\Omega,\bar{\cF},\bar{\cP})$ is a probability space.

Given $x\in\bbR^d$ let $z(x)\in\bbZ^d$ and $r(x)\in [0,1)^d$ be such that $x=z(x)+r(x)$ (they are univocally defined). Set $\bar\theta_x:\bar\Omega\to\bar\Omega$ as $\bar\theta _x (\omega,a):= (\theta _{z(x+a)}\omega,r(x+a))$. Then one can prove that $(\bar\theta_x)_{x\in\bbR^d}$ is a action of $\bbR^d$ on $\bar{\Omega}$ and $\bar{\cP}$ is stationary and ergodic for this action \rot{\cite[Section~6]{Fhom1}}.
In addition, define $\bar \mu_{(\o, a)} (\cdot) := \mu_\o( \cdot +a)$ for all $(\o,a)\in \bar \O$. Then, \rot{by \cite[Eq.~(61)]{Fhom1}}, we have again the covariant relation $\bar \mu_{\theta_x (\o,a)} = \bar \t_x \bar\mu_{(\o,a)}$, where $(\bar \t_x)_{x\in \bbR^d}$ is the action of the group $\bbR^d$ on the Euclidean space $\bbR^d$ by translations $\bar \t_x y:=\t_x= y+x$.
\rot{In conclusion,} the new setting given by the group  $\bar \bbG:=\bbR^d$,  the probability space  $(\bar\Omega,\bar{\cF},\bar{\cP})$, the actions $(\bar \theta _x)_{x\in \bar \bbG}$ and $(\bar \t_x )_{x\in \bar \bbG}$, the random measures $\bar \mu_{\bar \o}$ with $\bar \o \in \bar \O$ satisfies  Assumptions 1 and 2.

 Let us \rot{now} show that
  $\bar{\bbE} [ \bar\mu _{\bar \o}(\D)^\a]<+\infty$,
  where $\bar \bbE[\cdot]$ denotes the expectation w.r.t. $\bar\bbP$.
This follows from the bound  $\bbE [\mu _{\o}(\rrr{\D})^\a]<+\infty$ and the observation (based also on the stationarity of $\bbP$) that 
\begin{equation*}
\begin{split}
& \rrr{ \bar{\bbE} [ \bar\mu _{\bar \o}(\D)^\a] =\int _\O d\bbP(\o) \int_\D da \bar\mu_{(\o,a)}(\D)^\a=\int _\O d\bbP(\o) \int_\D da \mu_\o (\D+a)^\a}\\
  &\rrr{ =\int_\D da\bbE[\mu_\o (\D+a)^\a]=\int_\D da\bbE[\mu_\o (\D)^\a]=\bbE[\mu_\o (\D)^\a]}\,.
  \end{split}
\end{equation*}
In particular the new setting satisfies the assumptions of Lemma \ref{cinguetto_bis}. As a consequence,
 the result  obtained in Section \ref{polpette} implies that there exists a $\bar \bbG$--invariant measurable subset $\bar \cC\in \bar \cF$ with $\bar \bbP(\bar \cC)=1$  such that 
\[\lim_{\ell \uparrow\rot{+} \infty}\, \varlimsup_{\e \da 0} \int_{\{|x|\geq \ell\}} \vartheta (|x|)d \bar \mu_{( \o,a) }^\e (x) =0\]
for any $(\o,a) \in \bar \cC$.
Hence, by the definition of $\bar \mu^\e_{(\o,a)} $ and by   \eqref{pollice} and \eqref{pioli}, 
 for any  $(\o,a) \in \bar \cC$  it holds
\be\label{scivolino}
\lim_{\ell \uparrow \rot{+}\infty}\, \varlimsup_{\e \da 0}\e^d  \int_{\{|\e (x-a)|\geq \ell\}} \vartheta  (\e |x-a|)d \mu_{\o } (x) =0\,. \en

We now define $ \cC_*:=\{\o\in \O\,:\, \exists a\in \D \text{ with } (\o,a) \in \bar \cC\}$. In general, the projection of a measurable subset in a product measure space does not need to be  measurable. We can anyway show that $\cC_*\in \cF$ as follows. 
We know that $\bar \cC$ is measurable and $\bar\bbG$--invariant. Take $(\o,a)\in \bar \cC$ and $a'\in \D$. Then \rosso{$\bar \theta _{a'-a}(\o,a)=(\theta _{z(a')}\o , r(a'))=(\theta_0\o, a')= (\o,a')$}. This observation and the $\bar\bbG$-invariance of $\bar \cC$  imply  that $(\o,a')\in \bar \cC$ for any $(\o,a)\in \bar \cC$. Hence, $\bar \cC=\cC_* \times \D$. Since sections are measurable (see \cite[Exercise~1.7.18-(iii)]{Tao})   we conclude that $\cC_*$  is measurable, i.e.~$\cC_*\in \cF$.
Moreover, since $\bar \bbP(\bar \cC)=1$, it must be $\bbP(\cC_*)=1$ by Fubini-Tonelli Theorem.

At this point, to conclude the proof of Claim \ref{claim_pieno} (i.e. $\bbP(\cC)=1$), it remains to show that \rosso{$\cC_*\subset\cC$}. To this aim we take  $\o\in \cC_*$. We know that \eqref{scivolino} holds for \rosso{any} $a\in \D$.
\rosso{By taking $a=0$, \eqref{scivolino} reduces to \eqref{duomo}}.

\section{Proof of Theorem \ref{teo2}}\label{squillano}
We prove Items (i) and (ii) of Theorem \ref{teo2}. Item (i) trivially implies 
\rot{the limits \eqref{marvel0},...,\eqref{ondinoB} for $f\in C_*(\bbR^d)$}.
Let $\cB$ be as in Remark \ref{rinomato} and  Claim \ref{sale78} 
For Item (i) we define $d_c:=2d+2$  and $\O_{\rm typ}^*\in\cF$ as $\O^*_{\rm typ}:= \O_{\rm typ}\cap \cA\cap \cB\cap \cC$, where $\cA$ and $\cC$ are as in 
 Theorem \ref{teo_tartina} and Lemma \ref{cinguetto} respectively. For Item (ii) we define $d_c:=d$ and   $\O_{\rm typ}^*\in\cF$ as $\O^*_{\rm typ}:= \O_{\rm typ}\cap \cA\cap \cB\cap \cC$, where $\cA$ and $\cC$ are as
    in Corollary \ref{lollo_bronzo}.  Note that   $\O^*_{\rm typ}$ is  measurable, $\bbG$--invariant and $\bbP(\O_{\rm typ}^*)=1$.
 From now on we restrict to $\o \in \O^*_{\rm typ}$. \rot{Due to our definition of $d_c$, we need to prove \eqref{marvel0},...,\eqref{ondinoA} for all $f\in \cG(d_c/2)$, and \eqref{marvel2} and \eqref{ondinoB} for all $f\in \cG(d_c)$.}

\rot{The proof will use the following two  technical lemmas proved in Sections \ref{sec_auto1} and \ref{sec_auto2} respectively (the latter is similar to \cite[Lemma~6.1]{F1}):}

\begin{Lemma}\label{auto1}
 Let $f$ be a \rot{measurable function}. Let $t\geq 0$ and $\l>0$. If for some $C,\b>0$ it holds  $ | f(x)| \leq C (1+|x|)^{-\b}$ for all $x\in \bbR^d$,
  then for some $C'$ it holds $ |P_t f(x)| \leq C' (1+|x|)^{-\b}$ and $ |R_\l f(x)| \leq C' (1+|x|)^{-\b}$ for all $x\in \bbR^d$. 
\end{Lemma}

\begin{Lemma}\label{auto2}
Given  $h\in \rosso{\cG(d_c/2)}$ and given a family $h^\e_\o \in L^2(\mu^\e_\o)$ with $L^2(\mu^\e_\o)\ni  h^\e_\o  \to h \in L^2(m dx)$, it holds $\lim_{\e \da 0} \int_{\bbR^d} |h^\e_\o(x) - h(x) |^2 d\mu^\e_\o(x)=0$.
\end{Lemma}

\subsection{\rot{Proof of \eqref{marvel0}}} 
 \rot{The second limit in \eqref{fretta} can be restated as follows: if $L^2(\mu^\e_\o) \ni h_\e \mapsto h\in L^2(m dx)$ as $\e \da 0$, then $R_{\o,\l}^\e h_\e \to R_\l h$ as $\e \da 0$. As a consequence, by   \cite[Theorem~9.2]{ZP},
 one gets that it also holds $P_{\o,t}^\e h_\e \to P_t h$ as $\e\da 0$. Hence, to prove \eqref{marvel0}, it is enough to prove  the strong convergence  $L^2(\mu^\e_\o) \ni f \mapsto f\in L^2(m dx)$ as $\e \da 0$ for any  $f \in \cG(d_c/2)$.} 
 \rot{It is known (see e.g.~\cite[Remark~3.12]{Fhom1}) that  the strong convergence follows from the following limits as  $\e\da0$:
\be\label{amarena}
 \|f\|_{L^2(\mu^\e_\o)} \to \|f\|_{L^2(m dx)}\,, \qquad 
 L^2(\mu^\e_\o) \ni f  \rightharpoonup  f\in L^2(m dx)\,.
\en }
By Theorem \ref{teo_tartina} and  Corollary \ref{lollo_bronzo},  
  we get that  $f \in L^2(\mu^\e_\o)$ and  $\|f\|_{L^2(\mu^\e_\o)}^2=\int_{\bbR^d} f(x)^2 d\mu^\e_\o(x)$ converges to $\|f\|^2_{L^2(m dx)}=m \int _{\bbR^d} f(x)^2 dx$ as $\e\da 0$, thus proving the first limit in \eqref{amarena}.
 Since $f \psi\in C_c(\bbR^d)$ for any $\psi \in C_c(\bbR^d)$, the second limit in \eqref{amarena} follows \rot{from  Claim \ref{sale78}} and the definition of weak convergence.

\subsection{\rosso{Proof of \eqref{vinello}}} \rot{To prove that  $f\in \cG(d_c/2)$ satisfies \eqref{vinello}, it is enough to apply   \eqref{fretta}  where $u_\e= R^\e_{\o,\l}f$ and $u=R_\l f$ as we have already shown for \eqref{marvel0} the strong convergence  $L^2(\mu^\e_\o) \ni f \to f\in L^2(m dx)$ as $\e \da 0$.  }



\subsection{\rot{Proof of \eqref{marvel1} and \eqref{ondinoA}}}
Let $f\in \cG(d_c/2)$. We can apply  \rot{Lemma} \ref{auto2} with $h_\o^\e:= P^\e_{\o,t} f$ and $h=P_t f$. Indeed, we know that $P_t f\in \cG(d_c/2)$ by \rot{Lemma} \ref{auto1} and we know that  $L^2(\mu^\e_\o)\ni  h^\e_\o  \to h \in L^2(m dx)$ by \eqref{marvel0}. Then by \rot{Lemma}  \ref{auto2} we get \eqref{marvel1}. By the same arguments, and using now \eqref{vinello}, we get \eqref{ondinoA}.

\subsection{\rot{Proof of \eqref{ondinoB}}} \rot{We follow the main ideas of  \cite[Section~20]{Fhom1}, supplemented with the technical results developed here}.
Let $f\in \cG(d_c)$.   Without loss of generality \rot{we} can take $f\geq 0$. We fix $\b>d_c$ such that
$0\leq f(x) \leq C(1+|x|)^{-\b}$ for some $C>0$. By  \rot{Lemma} \ref{auto1}, \rrr{$0\leq R_\l f(x)\leq C'(1+|x|)^{-\b}$  for some $C'>0$}.
\rot{We set $B(n):=\{x\in \bbR^d\,:\,|x|\leq n\}$. By using Schwarz inequality to bound $ \| (R^\e_{\o,\l} f- R_\l f) \mathds{1}_{B(n)}\|_{L^1(\mu^\e_\o)}$}, we get
\be\label{visual}
  \begin{split}
  \|R^\e_{\o,\l} f- R_\l f\|_{L^1(\mu^\e_\o)}
&\leq
\| (R^\e _{\o,\l} f)  \mathds{1}_{B(n)^c} \|_{L^1(\mu^\e_\o)}
+ \| (R_\l f)\mathds{1} _{B(n)^c}\|_{L^1(\mu^\e_\o)}
    \\
    &+ \mu^\e_\o( B(n))^\frac{1}{2} \| R^\e_{\o,\l} f -R_\l f\|_{L^2(\mu^\e_\o)}\,.
  \end{split}
  \en
  The last addendum in the r.h.s. of \eqref{visual} goes to zero as $\e\da 0$ by \eqref{ondinoA} and since 
  $ \mu^\e_\o( B(n)) \to m\rrr{\ell (B(n))}$  as $\e \da 0 $ by Claim \ref{sale78}.
For the second addendum we claim that 
$\lim _{n \uparrow +\infty} \varlimsup_{\e \da 0}\|\rrr{(R_\l f)} \mathds{1}_{B(n)^c} \|_{L^1(\mu^\e_\o)}=0$. Indeed,  since  
\be\label{damasco}
  \|\rrr{(R_\l f)} \mathds{1}_{B(n)^c} \|_{L^1(\mu^\e_\o)}\leq C' \int_{B(n)^c }\rrr{ (1+|x|)^{-\b}} d\mu^\e_\o(x)\,,
\en
 one can apply Lemma \ref{cinguetto} and Corollary  \ref{lollo_bronzo} to get that the r.h.s. of \eqref{damasco} goes to zero as $\e\da 0$ and afterwards $n\uparrow +\infty$.
 
\rot{Finally let us show that  the first addendum in the  r.h.s. of \eqref{visual} is negligible as $\e \da 0$ and afterwards $n\uparrow +\infty$. Since $f\geq 0$ we can write
\be\label{summer1}
\| (R^\e _{\o,\l} f)  \mathds{1}_{B(n)^c} \|_{L^1(\mu^\e_\o)} = \|  R^\e _{\o,\l} f\| _{L^1(\mu^\e_\o)}- \| (R^\e _{\o,\l} f)  \mathds{1}_{B(n)} \|_{L^1(\mu^\e_\o)}
\,.
\en A similar formula holds for $R_\l$. The proof is then completed if we show that 
\begin{align}
& \lim_{\e \da 0} \|  R^\e _{\o,\l} f\| _{L^1(\mu^\e_\o)}=  \| R_\l f \| _{L^1 (m dx) }\,,\label{summer2} \\
&  \lim_{\e \da 0} \| (R^\e _{\o,\l} f)  \mathds{1}_{B(n)} \|_{L^1(\mu^\e_\o)}=  \| (R_\l f)   \mathds{1}_{B(n)}  \| _{L^1 (m dx) } \,.\label{summer3}
\end{align}
Indeed, by combining \eqref{summer1}, \eqref{summer2} and \eqref{summer3}, we get $\lim_{\e \da 0} \| (R^\e _{\o,\l} f)  \mathds{1}_{B(n)^c} \|_{L^1(\mu^\e_\o)} =  \| (R_\l f)   \mathds{1}_{B(n)^c}  \| _{L^1 (m dx) }$
and the last expression goes to zero as $n\uparrow +\infty$ by the integrability of $R_\l f$.  }

\rot{Let us prove \eqref{summer2}. Since $f\geq 0$ and by the reversibility of $\mu^\e_\o$ for the diffusively rescaled random walk due to  Item (iii) of Assumption 4,  we have  \[
\|  R^\e _{\o,\l} f\| _{L^1(\mu^\e_\o)}=\int_0^\infty ds \,e^{-\l s} \int _{\bbR^d} d \mu^\e_\o(x) P^\e_{\o,s} f (x)= \l^{-1}\int _{\bbR^d} d \mu_\o^\e (x) f(x) \,.
\]
As $f\in \cG(d_c)$, by  Theorem \ref{teo_tartina} and Corollary \ref{lollo_bronzo}, the last term  converges to $ \l^{-1}m \int _{\bbR^d} f(x)d x= \| R_\l f \| _{L^1 (m dx) }$ as $\e \da 0$. This completes the proof of \eqref{summer2}.}

\rot{Let us prove  \eqref{summer3}. Since  $f\geq 0$ we have
\be
\begin{split}
&\left |  \| (R^\e _{\o,\l} f)  \mathds{1}_{B(n)} \|_{L^1(\mu^\e_\o)}-  \| (R_\l f)   \mathds{1}_{B(n)}  \| _{L^1 (m dx) } 
 \right |\\
&  =\left| \la R^\e _{\o,\l} f- R_\l f  , \mathds{1}_{B(n)} \ra_{L^2(\mu^\e_\o)}\right|
 \leq \mu^\e_\o(B_n)^{1/2}  \| R^\e _{\o,\l} f- R_\l f \| _{L^2(\mu^\e_\o)}
 \,.
\end{split}
\en
The r.h.s. goes to zero as $\e \da 0$ being equal to the last term in \eqref{visual}.}


\subsection{\rot{Proof of \eqref{marvel2}}} Let $f\in \cG(d_c)$.   
In light of the \rot{arguments used for \eqref{ondinoB}}, one can similarly derive \eqref{marvel2} from \eqref{marvel1} (similarly to the proof of \cite[Corollary~2.5]{F1}).

\subsection{Proof of Lemma  \ref{auto1}}\label{sec_auto1}
\rot{In what follows, given $a\in \bbR^m$,  $|a |$ will denote its \ciak{Euclidean} norm  in $\bbR^m$.  Moreover, below constants of type $c$ are  positive, depend only on $d,\b,\l$  and can change from line to line. If a constant  can depend also from  $t$ in addition to $d,\b,\l$ , we  write $c(t)$. }

\rot{Using that $\|P_t f\|_\infty\leq \|f\|_\infty$ and $\| R_\l f \|_\infty \leq \|f\|_\infty /\l$, it is enough to prove the bounds on $|P_t f(x)|$ and  $|R_\l f(x)|$ for $x$ large. We will use this observation later.}

\rot{Since  $D$ is a positive semidefinite symmetric matrix, it can be diagonalized  by an orthonormal basis of $\bbR^d$. 
Let us call $k$  the dimension of ${\rm Ker}(D)^\perp$. If $k=0$, then $P_t f(x) = f(x)$ and $R_\l f(x)= f(x)/\l$ and  trivially the lemma is  valid in this case. Let us now take $k \geq 1$. Recall that $P_t$ is the Markov semigroup associated to the Brownian motion on $\bbR^d$ with diffusion matrix $2D$. If for example $D e_i=1$ for $1\leq i \leq k$ and $D e_i=0$ for $k<i\leq d$, then  we have 
\be\label{flauto25}
P_t f(w,z) := \frac{1}{ (4 \pi t )^{k/2} }
\int  _{\bbR^k}    e^{  -       \frac{|w-y|^2}{4 t}    } f(y,z) dy\; \;\;\;\; \forall (w,z) \in \bbR^k\times \bbR^{d-k}=\bbR^d\,.
\en
The general form of $P_t f$ is similar to \eqref{flauto25}  in a different orthonormal basis.}

 \rot{In addition note that by our assumptions   and due to \eqref{flauto25}, we have $|P_t f(w,z) | \leq C (4 \pi t )^{-k/2} 
\int  _{\bbR^k}    e^{  -       \frac{|w-y|^2}{4 t}    }  (1+|(y,z)|) ^{-\b} dy$.  Let us set
\[   A(w,z,t):= t^{-\frac{k}{2}} \int_{\bbR^k} e^{-\frac{|w-y|^2}{t}}(1+ |(y,z) |)^{-\b}dy\; \;\;\;\; \forall (w,z) \in \bbR^k\times \bbR^{d-k}=\bbR^d\,.\]
We write $A(x,t)$ for $A(w,z,t)$ when $x=(w,z)$. Then, apart from  some rescaling and some multiplicative constants,  the lemma is proved if we show that,  for some $c(t),c$ and  for all $x\in \bbR^d$ with $|x|  $  large enough,  it holds 
\be\label{caldo} | A(x,t)| \leq c(t) (1+ |x|)^{-\b} \;\text{ and }\; \int_0^\infty e^{-\l t}| A(x,t)| dt\leq  c (1+ |x |)^{-\b} \,.
\en
}

\rot{We first point out that \eqref{caldo}  can be easily derived when  $x=(w,z)\in \bbR^k\times \bbR ^{d-k}$ and $|w|\leq |z|$. Indeed, it is enough to use in the definition of $A(w,z,t)$ that  $ (1+|(y,z)|) ^{-\b}\leq (1+|z|)^{-\b}\leq  c (1+|x|)^{-\b}$
since 
 $ |x| \leq \sqrt{2} |z|$. 
}

 
 \rot{Let us consider now the case $x=(w,z)$ with $|w| >|z|$.  In this case $|w|\leq |x|\leq \sqrt{2}|w|$.  Using in the definition of $A(w,z,t)$ that 
 $ (1+|(y,z)|) ^{-\b}\leq (1+|y|)^{-\b}$  and at the end that
  $(1+|w|)^{-\b} \leq c (1+|x|)^{-\b}$, to prove \eqref{caldo} it is enough to show  for  $w\in \bbR^k$ and $|w|$ large enough (let us say for  $|w|\geq 1$) that 
 \be\label{caldobis} | C(w,t)| \leq  c(t) (1+ |w|)^{-\b} \;\text{ and }\; \int_0^\infty e^{-\l t}| C(w,t)| dt\leq  c(1+ |w |)^{-\b} \,,
\en
where $C(\rot{w},t):= t^{-\frac{k}{2}} \int_{\bbR^k} e^{-\frac{|\rot{w}-y|^2}{t}}(1+ |y|)^{-\b}dy$. Let us set  }
 \begin{align*}
    & C_1(\rot{w},t):=t^{-\frac{k}{2}} \int_{\{|y|\leq |\rot{w}|/2\} } e^{-\frac{|\rot{w}-y|^2}{t}}(1+ |y|)^{-\b}dy\,,\\
     & C_2(\rot{w},t):=t^{-\frac{k}{2}} \int_{\{|y|> |\rot{w}|/2\} } e^{-\frac{|\rot{w}-y|^2}{t}}(1+ |y|)^{-\b}dy\,.
  \end{align*}
  Then $C(\rot{w},t)=C_1(\rot{w},t)+C_2(\rot{w},t)$ and 
  \begin{align}
   &  C_1(\rot{w},t)\leq t^{-\frac{k}{2}}
     e^{-\frac{|\rot{w}|^2}{4 t}} \int_{\{|y|\leq |\rot{w}|/2\} } 1 dy\leq       c t^{-\frac{k}{2}}  e^{-\frac{|\rot{w}|^2}{4 t}} (1+|\rot{w}|)^k\,,\label{lillo1}
    \\
  & C_2(\rot{w},t)\leq (1+ |\rot{w}|/2)^{-\b} t^{-\frac{k}{2}} \int_{\bbR^k} e^{-\frac{|\rot{w}-y|^2}{t}}d\rot{y}\leq c (1+ |\rot{w}|)^{-\b} \,.\label{lillo2}
  \end{align}
 \rot{The estimates \eqref{lillo1} and \eqref{lillo2} imply the first bound in \eqref{caldobis}}.

     We move to the \rot{second bound in \eqref{caldobis}. Recall that $|w|\geq 1$.}  
     We can bound $  \int _0^\infty   e^{-\l t} t^{-\frac{k}{2}}  e^{-\frac{|\rot{w}|^2}{4 t}}dt $ by treating separately the  integral over $\{ t>|\rot{w}|\}$  and the integral over
 $\{0 \leq t\leq |\rot{w}|\}$. The former is bounded by $\int_{|\rot{w}|}^\infty e^{-\l t}dt =
e^{-\rot{\l |w|}}/\l$. The latter is bounded by 
$  e^{-\frac{|\rot{w}|}{8 }}  \int_0^{|\rot{w}|} t^{-\frac{k}{2}} e^{-\frac{|\rot{w}|^2}{8 t}}  dt$,  
\rot{which   is} bounded by 
\[   e^{-\frac{|\rot{w}|}{8 }}  |\rot{w}|^{-k} \int_0^{|\rot{w}|}
  (t/|\rot{w}|^2)^{-\frac{k}{2}} e^{-\frac{|\rot{w}|^2}{8 t}}  dt
  \leq c  e^{-\frac{|\rot{w}|}{8}}|\rot{w}|^{1-\rot{k}}
  \]
 (we have used that $s^{-\frac{k}{2}}e^{ -\frac{1}{8s}}$ is bounded  on $\bbR_+$).
  \rot{The above bounds imply that  $ \int _0^\infty   e^{-\l t} t^{-\frac{k}{2}}  e^{-\frac{|\rot{w}|^2}{4 t}}dt  \leq c e^{-c' |w|}$ for $|w| \geq 1$}.
By combining \rot{this last estimate} with \eqref{lillo1} and \eqref{lillo2}, we get 
$\int _0^\infty e^{-\l t} C(\rot{w},t) dt \leq c (1+|\rot{w}|)^{-\b}$ for $|\rot{w}|\geq 1$. 

\subsection{Proof of Lemma \ref{auto2}}\label{sec_auto2} \rot{The proof is similar to the proof of   \cite[Lemma~6.1]{F1}, where we considered  random walks on $\bbZ^d$ for which  it was  immediate to check  the ergodic theorems  with weights proved here. We give the proof to make clear where the ergodic issues  play an important role.}

  Since $\rosso{h\in \cG(d_c/2)}$ \rot{and therefore $h^2\in \cG(d_c)$,}  by Theorem \ref{teo_tartina} and \rosso{Corollary \ref{lollo_bronzo}},  we have $\int _{\bbR^d} h(x)^2 d\mu^\e_\o(x) \to m \int _{\bbR^d} h(x)^2 dx$ as $\e \da 0$.
Let us show that \rot{also}  $\int _{\bbR^d} h^\e_\o (x)^2 d\mu^\e_\o(x) \to m \int _{\bbR^d} h(x)^2 dx$ \rot{as $\e \da 0$}.  To get this limit it is enough to observe that, 
 by Remark \ref{rinomato},   $h^\e_\o \rightharpoonup h$ and this allows to take 
 $g_\e:=h^\e_\o$ and $g:=h$ when applying  Definition \ref{debole_forte} to the strong convergence 
 $L^2(\mu^\e_\o)\ni  h^\e_\o  \to h \in L^2(m dx)$. This implies that $\int _{\bbR^d} h^\e_\o (x)^2 d\mu^\e_\o(x) \to m \int _{\bbR^d} h(x)^2 dx$.

As $\int _{\bbR^d} h(x)^2 d\mu^\e_\o(x) \to m \int _{\bbR^d} h(x)^2 dx$ and
 $\int _{\bbR^d} h^\e_\o (x)^2 d\mu^\e_\o(x) \to m \int _{\bbR^d} h(x)^2 dx$, 
to conclude the proof of our claim
   we  show  that $ \int_{\bbR^d} h^\e_\o(x)  h(x)  d\mu^\e_\o(x)\to m \int_{\bbR^d} h(x)^2 dx$.
   Since $h \in \cG(d_c/2)$ we have $|h(x) | \leq C (1+|x|)^{-\b}$  for all $x\in \bbR^d$, for some $C>0$ and  $\b >d_c/2$.
Given an integer $\ell>0$ we fix a function $g_\ell\in C_c(\bbR^d)$ such that $g_\ell (x)= h(x)$ for $|x| \leq \ell$ and $|g_\ell(x) | \leq C (1+|x|) ^{-\b}$ for all $x\in \bbR^d$.
    Then, using the weak convergence  $h^\e_\o \rightharpoonup h$, we get
    \begin{equation*}
    \lim_{\e \da 0}  \int_{\bbR^d} h^\e_\o(x)  g_\ell(x)  d\mu^\e_\o(x)= m \int_{\bbR^d}h(x)  g_\ell (x)  dx\,,
    \end{equation*}
    which trivially implies that
\begin{equation}\label{siria1}
   \lim_{\ell \uparrow+\infty} \lim_{\e \da 0}  \int_{\bbR^d} h^\e_\o(x)  g_\ell(x)  d\mu^\e_\o(x)= m \int_{\bbR^d}h(x)^2  dx\,.
    \end{equation}
    On the other hand, by Schwarz inequality and our bounds on $h$ and $g_\ell$ and  recalling that $B(\ell):=\{x\in \bbR^d\,:\,|x|\leq \ell\}$, we can estimate
   \begin{multline}\label{cavallino}
      \Big| \int_{\bbR^d} h^\e_\o(x) \bigl(h(x)- g_\ell(x)\bigr)  d\mu^\e_\o(x)\Big| \leq \| h^\e_\o \|_{L^2(\mu^\e_\o)}  \| h(x)- g_\ell(x)\|_{L^2(\mu^\e_\o)}  \\
                                                                                    \leq 2  C  \| h^\e_\o \|_{L^2(\mu^\e_\o)}\|  \mathds{1}_{\rot{B(\ell)^c}} (x) (1+|x|)^{-\b} \|_{L^2(\mu^\e_\o)}\,.
                                                                                    \end{multline}
 \rot{Since $2\b>d_c$ and due to Lemma \ref{cinguetto} and Corollary \ref{lollo_bronzo}, we have that 
 \[\varlimsup_{\ell\uparrow +\infty }\lim _{\e \da 0}\|  \mathds{1}_{\rot{B(\ell)^c}} (x) (1+|x|)^{-\b} \|^2_{L^2(\mu^\e_\o)}
 =\varlimsup_{\ell\uparrow +\infty }\lim _{\e \da 0} \int _{\{|x| \geq \ell\}  }(1+|x|)^{-2\b} d\mu^\e_\o(x) =0\,.
 \]
 }
    By the above estimate, \rot{\eqref{cavallino} and}
    since $\varlimsup_{\e \da 0}  \| h^\e_\o \|_{L^2(\mu^\e_\o)}<+\infty$ by  \rot{our} assumption that \rot{$h^\e_\o \to h$},  we conclude that
    \be\label{siria2}
    \lim_{\ell\uparrow+\infty}\varlimsup_{\e \da 0}
    \Big| \int_{\bbR^d} h^\e_\o(x) \bigl(h(x)- g_\ell(x)\bigr)  d\mu^\e_\o(x)\Big|=0\,.
    \en
  To  \rot{get that  $\int_{\bbR^d} h^\e_\o(x) h(x)  d\mu^\e_\o(x) \to 
 m \int_{\bbR^d} h(x)^2 dx$ as $\e\da 0$,}  it is  now enough to combine \eqref{siria1} and \eqref{siria2}.

\appendix

\section{Proof of Claim \ref{sale78}}\label{WT1} We   take $\bbG=\bbR^d$ (the case $\bbG=\bbZ^d$ can be done by similar arguments).  Let $\cB\subset \O$ be  the measurable set given by the $\o\in\O$ for which the limit \eqref{piazza_fiume} holds with  $\varphi$ replaced by any indicator function $\mathds{1}_{(a,b]}$, where $(a,b]:= \prod_{i=1}^d (a_i,b_i]$,  $a,b\in \bbQ^d$, $a_i<b_i$ for  $i=1,\dots,d$. 
  By 
  \cite[Theorem~10.2.IV]{DV} (based on  the  results in \cite{T1972})
  we get that $\bbP(\cB)=1$.
  Since any function in  $ C_c(\bbR^d)$ with support in some  ball $B(r)$ can be approximated from above and from below by linear combinations of indicator functions   $\mathds{1}_{(a,b]}$ as above with $(a,b]\subset B(r+1)$ and since the approximation can be done  with arbitrarily small error in uniform distance, it is simple to get that  \eqref{piazza_fiume} holds  for any  $\o \in \cB$ and any function in $C_c(\bbR^d)$.

  Let us prove that $\cB$ is $\bbG$--invariant. Take $\o\in \cB$.  \rot{Recall that $\mu_{\theta_g \o}=\t_g \o$ by Assumption $2$.} Since by \eqref{pollice} and \eqref{pioli}
  \[ 
    \int_{\bbR^d}  f(x)d \mu^\e _{\theta_g\o}(x) = \e^d \int_{\bbR^d}f( \e(\t_{-g}y)) d\mu_\o(y)=  \e^d \int_{\bbR^d}f( \e y - \e V g) d\mu_\o(y)\]
  and $f\in C_c(\bbR^d)$ is uniformly continuous, we get that
  $\int_{\bbR^d}  f(x)d \mu^\e _{\theta_g\o}(x)\to m \int_{\bbR^d} f(x) dx $ as $\e \da 0$ for any $f\in C_c(\bbR^d)$ (we use that \eqref{piazza_fiume} holds for functions in $C_c(\bbR^d)$ since $\o\in\cB$). By suitably approximating by functions in $C_c(\bbR^d)$    any indicator function $\mathds{1}_{(a,b]}$ with $a,b$ as  above, we then get that   $\int_{\bbR^d}  \mathds{1}_{(a,b]}(x)d \mu^\e _{\theta_g\o}(x)\to m \int_{\bbR^d}  \mathds{1}_{(a,b]}(x)  dx $ as $\e \da 0$ and for all $a,b\in \bbQ^d$ with $a_i < b_i$. Hence $\theta_g\o\in \cB$. This concludes  the proof that 
  $\cB$ is $\bbG$--invariant.

 \medskip

\noindent
{\bf Acknowledgements}.  I thank  Prof.~A. Tempelman for useful discussions. \rot{I thank the anonymous \ciak{referees} for stimulating  comments and  corrections. As  stated at the beginning, I warmly thank Francis Comets}

\end{document}